\documentclass[11pt,letterpaper]{amsart}
\usepackage{amsmath,amssymb,microtype}
\usepackage[colorlinks=true,linkcolor=blue,citecolor=blue,urlcolor=blue]{hyperref}
\allowdisplaybreaks[2]
\numberwithin{equation}{section}

\newtheorem{theorem}{Theorem}[section]
\newtheorem{proposition}[theorem]{Proposition}
\newtheorem{lemma}[theorem]{Lemma}
\newtheorem{remark}[theorem]{Remark}

\newcommand{\R}{\mathbb R}
\newcommand{\bu}{\mathbf u}
\newcommand{\bv}{\mathbf v}
\newcommand{\bw}{\mathbf w}
\newcommand{\bh}{\mathbf h}
\newcommand{\cN}{\mathcal N}
\newcommand{\Sn}{\mathbb S^{n-1}}
\newcommand{\Ftwo}{\mathbf F_2}
\newcommand{\qtwo}{\mathbf q}
\newcommand{\bA}{\mathbf A}

\title[Optimal regularity for a singular cooperative system]
{Optimal regularity of solutions to a singular cooperative system}

\author{Lili Du}
\address{\textsc{Lili Du:}\newline
	Department of Mathematics, Sichuan University,
	Chengdu 610064, P.\,R.\,China.}
\email[L. Du]{dulili@scu.edu.cn}

\author{Xu Tang}
\address{\textsc{Xu Tang$^{\ast}$:}\newline
	School of Mathematical Sciences, Fudan University,
	Shanghai 200433, P.\,R.\,China.}
\email[X. Tang]{tang\_xu@fudan.edu.cn}

\author{Cong Wang}
\address{\textsc{Cong Wang:}\newline
	School of  Mathematics, Southwest Jiaotong University,
	Chengdu  610031, P.\,R.\,China.}
\email[C. Wang]{CongWang@swjtu.edu.cn}

\thanks{$^{\ast}$Corresponding author: Xu Tang (tang\_xu@fudan.edu.cn)}

\date{}

\begin{document}

\begin{abstract}
	In the pioneer work of vectorial free boundary problems \cite[Adv. Math. 280 (2015)]{ASUW15}, Andersson, Shahgholian, Uraltseva, and Weiss investigated the regularity theory of solutions and the free boundary of a singular cooperative system
	\[
	\Delta\bu=\frac{\bu}{|\bu|}\chi_{\{|\bu|>0\}}
	\quad\text{in }\Omega\subset\R^n,
	\qquad
	\bu:\Omega\longrightarrow\R^m,
	\qquad n,m\ge2.
	\]
	At the level of solutions, the boundedness of the right-hand side directly yields $W^{2,p}_{\mathrm{loc}}$ regularity for every $1<p<\infty$ by standard elliptic estimates, while the corresponding $W^{2,\infty}_{\mathrm{loc}}$ regularity question was explicitly left as an open problem on page~753 of that paper. In this paper, we give a positive answer by establishing optimal $W^{2,\infty}_{\mathrm{loc}}$ regularity through a scale-invariant interior Hessian estimate for every weak solution. Unlike the scalar case, the classical Alt--Caffarelli--Friedman monotonicity formula cannot be applied directly to the present system. Our proof instead combines a localized Newtonian potential decomposition with an explicit projection onto quadratic harmonic polynomials on the unit sphere, which isolates the symmetric trace-free quadratic part of the Hessian. A key new observation is that the resulting matrix coefficients satisfy an exact linear ordinary differential equations with constant coefficients on the logarithmic scale. The main novelty of this work lies in a so-called two-regime argument based on the competition between the affine and quadratic parts of the rescaled solution, together with a continuity argument that connects the two regimes and yields a uniform bound for the coefficient. The resulting a priori estimate provides an analytic basis for further study of the fine structure of the free boundary.
\end{abstract}


\maketitle

\setcounter{tocdepth}{1}
\tableofcontents

\section{Introduction}

The pioneering work of Andersson, Shahgholian, Uraltseva, and Weiss
\cite{ASUW15} introduced the singular cooperative system
\begin{equation}\label{eq:system}
	\Delta\bu=\frac{\bu}{|\bu|}\chi_{\{|\bu|>0\}}
	\quad\text{in }\Omega\subset\R^n,
	\qquad
	\bu:\Omega\longrightarrow\R^m,
	\qquad n,m\ge2,
\end{equation}
where $\Omega$ is an open set, $|\cdot|$ is the Euclidean norm, and
$\chi_E$ denotes the characteristic function of a measurable set $E$.
The system may be viewed as a particular equilibrium state of a cooperative
reaction--diffusion model.  In the two-species interpretation described in
\cite{ASUW15}, the components represent the concentrations of two reactants,
and the interaction is cooperative in the sense that each reactant slows the
extinction or reaction of the other.  For the free boundary analysis, the key
additional observation is that solutions of \eqref{eq:system} arise as
minimizers of a convex energy.  More precisely, each solution $\bu$ minimizes the energy
\begin{equation*}
	\mathcal E(\bu;\Omega)=\int_{\Omega}\bigl(|D\bu|^2+2|\bu|\bigr)\,dx.
\end{equation*}
Working from
this variational perspective, the authors of \cite{ASUW15} combined monotonicity
formulas, an epiperimetric inequality, and geometric arguments to establish
quadratic growth and nondegeneracy for energy-minimizing solutions, as well as
$C^{1,\alpha}$ regularity of the regular part of the free boundary
$\partial\{|\bu|>0\}\cap\Omega$.  Although these results laid the foundations
for a geometric theory of the system, the distinct analytic problem of
establishing optimal interior regularity for the map itself remained open.

At the level of arbitrary weak solutions, the boundedness of the right-hand
side provides a natural starting point for the regularity analysis.  Indeed,
the Calder\'on--Zygmund estimates immediately yield
\begin{equation}\label{eq:finite-p}
\bu\in W^{2,p}_{\mathrm{loc}}(\Omega;\R^m)
\quad\text{for every }1<p<\infty.
\end{equation}
However, these estimates do not extend to $p=\infty$, and the corresponding
optimal regularity question was formulated in \cite[p.~753]{ASUW15} as
follows.
\begin{center}
	\begin{minipage}{0.88\textwidth}
		\centering
		\emph{``Note that in contrast to the classical (scalar) obstacle problem,
			it is an open problem whether
			$\bu\in W^{2,\infty}_{\mathrm{loc}}(\Omega;\R^m)$.''}
	\end{minipage}
\end{center}
The same question was also mentioned in the very recent work
\cite[p.~4]{DJS26}, where it was stated in the following terms.
\begin{center}
	\begin{minipage}{0.88\textwidth}
		\centering
		\emph{``unlike in the classical scalar obstacle problem, it
			remains an open problem whether solutions belong to
			$W^{2,\infty}_{\mathrm{loc}}$.''}
	\end{minipage}
\end{center}
The main goal of this paper is to resolve this open problem by establishing
optimal $W^{2,\infty}_{\mathrm{loc}}$ regularity.

We first examine the scalar case $m=1$, in which the one-dimensional target counterpart of
\eqref{eq:system} is the two-phase membrane equation
\begin{equation*}
	\Delta u=\chi_{\{u>0\}}-\chi_{\{u<0\}} \quad\text{in }\Omega,
\end{equation*}
and Uraltseva \cite{Ura01} established successfully optimal $W^{2,\infty}_{\mathrm{loc}}$ regularity by
applying the Alt--Caffarelli--Friedman monotonicity formula to the positive and
negative parts of directional derivatives  (see also
\cite{Sha03}).  This formula first appeared in the pioneering work
\cite{ACF84} on a two-phase free boundary problem and has since become one of
the most powerful tools in free boundary analysis. The free boundary of this two-phase model was subsequently
studied in \cite{SUW04,SUW07}.
Moreover, under the additional nonnegativity constraint, it
reduces to the classical obstacle problem
\begin{equation*}
	\Delta u=\chi_{\{{u>0}\}}, \qquad u\ge0 \quad\text{in }\Omega.
\end{equation*}
As emphasized by Figalli \cite{Fig18}, a central goal is to understand
both the regularity of solutions and the structure of the contact set with
the obstacle.  Concerning the first aspect, Frehse \cite{Fre72} established the optimal $W^{2,\infty}_{\mathrm{loc}}$
regularity
in 1972.  The second aspect naturally leads to the analysis of the
free boundary, which is the relative boundary of the contact set in $\Omega$.
The regularity theory for its regular part was developed in significant works
\cite{Caf77,KN77,Caf80,Caf98}, whereas the finer structure of singular free
boundary points was established in \cite{Wei99,Mon03,CSV18,FS19}.  We also
refer to the monograph \cite{PSU12} and the surveys \cite{RS19,Dan20} for
broader accounts of these developments.

However, these scalar arguments do NOT apply directly to the vectorial system
\eqref{eq:system}.  It should be noted that in the vectorial case $m\ge2$, a
directional derivative of $\bu$ is vector-valued and has no canonical
positive and negative parts,
while taking such parts componentwise does not produce closed scalar
inequalities, since differentiation of $\bu/|\bu|$ away from the  zero
set couples all components through coefficients that become singular as
$|\bu|\to0$.  Consequently, as already observed in
\cite[p.~744]{ASUW15}, the disjoint scalar pair required by the
Alt--Caffarelli--Friedman formula appears to be unavailable for
\eqref{eq:system}.  This failure of scalar phase separation marks a
fundamental distinction between the scalar and vectorial problems and
constitutes one of the principal difficulties in the analysis of the
vector-valued free boundary problem. More broadly, the present system forms part of a developing theory of
vector-valued free boundary problems, and several related models have been
investigated in \cite{FSW21,DJS22,DJS23,FK24,DJS26}.

\subsection{The main result}

Before stating the main result, we fix the notation used throughout the paper.
We write $B_r(x)=\{y\in\R^n:|y-x|<r\}$ and $B_r=B_r(0)$, while
$\Omega'\Subset\Omega$ means that $\overline{\Omega'}$ is a compact subset of
$\Omega$.  Throughout the paper, $C_n>0$ denotes a constant that depends only
on $n$ and may change from one occurrence to the next.  Unless stated
otherwise, every $L^\infty$ norm is interpreted in the essential-supremum
sense.  To formulate the weak equation concisely, we introduce the bounded
normalization map
\begin{equation}\label{eq:N}
	\cN(z):=
	\begin{cases}
		z/|z|,&z\ne0,\\
		0,&z=0,
	\end{cases}
	\qquad z\in\R^m,
\end{equation}
so that \eqref{eq:system} becomes $\Delta\bu=\cN(\bu)$, where
$\cN=(\cN_1,\ldots,\cN_m)$ and $|\cN(z)|\le1$ for every $z\in\R^m$.  Thus, we call
$\bu=(u_1,\ldots,u_m)\in W^{1,2}_{\mathrm{loc}}(\Omega;\R^m)$ a weak
solution of \eqref{eq:system} if, for every
$\alpha\in\{1,\ldots,m\}$ and every $\varphi\in C_c^\infty(\Omega)$, it holds that
\begin{equation*}
	-\int_\Omega D u_\alpha\cdot\nabla\varphi\,dx
	=\int_\Omega\cN_\alpha(\bu)\varphi\,dx.
\end{equation*}
Since the Laplacian and all derivatives are taken componentwise,
$D^2\bu$ is the ordered $m$-tuple of component Hessians, and
\begin{equation*}
	|D^2\bu|^2
	=\sum_{\alpha=1}^m\sum_{i,j=1}^n
	|\partial_{ij}u_\alpha|^2.
\end{equation*}

Our main result is the following scale-invariant interior
estimate.

\begin{theorem}[Optimal regularity estimate]\label{thm:main}
	Let $n,m\ge2$ be integers, let $\bu$ be a weak solution of
	\eqref{eq:system} in $\Omega\subset\R^n$, and suppose that
	$B_{2R}(x_0)\Subset\Omega$ for some $R>0$.  Then
	\begin{equation}\label{eq:main-estimate}
		\|D^2\bu\|_{L^\infty(B_R(x_0))}
		\le C_n\left(1+R^{-2}
		\|\bu\|_{L^\infty(B_{2R}(x_0))}\right),
	\end{equation}
	where $C_n>0$ depends only on the domain dimension $n$.  In particular,
	every weak solution of system \eqref{eq:system}  belongs to
	$W^{2,\infty}_{\mathrm{loc}}(\Omega;\R^m)$.
\end{theorem}

\begin{remark}
	Theorem~\ref{thm:main} solves the open problem
	posed by Andersson, Shahgholian, Uraltseva, and Weiss in
	\cite[p.~753]{ASUW15} for every $n,m\ge2$.  Moreover,
	its regularity conclusion is optimal.  Indeed, the half-space solutions
	$\bu_{e,\nu}(x)=\frac12(x\cdot\nu)_+^2e$, where $e\in\R^m$,
	$\nu\in\R^n$, and $|e|=|\nu|=1$, belong to
	$C^{1,1}$, whereas their Hessians are discontinuous across the corresponding
	free boundaries.  Consequently, the $W^{2,\infty}_{\mathrm{loc}}$
	conclusion cannot, in general, be strengthened to $C^2$ regularity for
	arbitrary weak solutions of system \eqref{eq:system}.
\end{remark}

\begin{remark}
It should be noted that	the estimate \eqref{eq:main-estimate} is invariant under the natural
	quadratic rescaling
	\(
	\bu_{x_0,R}(x):=R^{-2}\bu(x_0+Rx).
	\)
	Indeed, the rescaled map satisfies the same system, its Hessian acquires no
	scaling factor, and its $L^\infty(B_2)$ norm is precisely
	$R^{-2}\|\bu\|_{L^\infty(B_{2R}(x_0))}$.  Thus the estimate for the
	rescaled map on $B_1$ is equivalent to \eqref{eq:main-estimate}.
\end{remark}

\begin{remark}
	It should also be emphasized that the constant $C_n$ in
	\eqref{eq:main-estimate} depends only on the domain dimension $n$ and is
	independent of the target dimension $m$.  The same proof is still available for 
	$m=1$, in which case it recovers the optimal estimate for the scalar
	two-phase membrane equation discussed above.  Theorem~\ref{thm:main} is
	stated for $m\ge2$ in order to emphasize the genuinely vectorial range, where
	the scalar phase-separation argument is unavailable.
\end{remark}

As discussed above, the scalar phase separation underlying the
Alt--Caffarelli--Friedman formula has no known analogue for the vector-valued
system.  Instead, we project the quadratically rescaled solution in $L^2$
onto the restrictions to the unit sphere $\mathbb S^{n-1}\subset\R^n$ of
vector-valued homogeneous harmonic polynomials of degree two.  Quadratic
harmonic projections have previously been used in scalar free boundary
problems \cite{ASW10,ALS13,IMN17}, while a related trace-free matrix projection
appears in \cite{CLW26}.  In the present system, the Laplacian determines the
trace of each component Hessian, whereas this single vector-valued projection
records all remaining trace-free quadratic coefficients simultaneously.
Accordingly, the optimal regularity problem reduces to controlling a
finite-dimensional coefficient along logarithmic scales, although identifying
this coefficient is only the starting point of the argument.

The main novelty in this paper lies in the uniform control of the vector-valued
quadratic coefficient, rather than in the projection itself.  The normalized
nonlinearity makes the projected right-hand side integrable when the
affine part dominates, whereas, when the quadratic part is
sufficiently large and dominates both the affine contribution and the bounded
remainder, it forces the coefficient to decrease.  Although neither quantitative
condition needs to hold at every scale, continuity connects the two estimates and
closes the intermediate range.  This combination replaces scalar phase
separation and yields the optimal Hessian estimate without classifying blow-ups
or requiring prior geometric information about the free boundary.

Moreover, the scope of the argument becomes clearer when its general and
equation-specific parts are separated.  The quadratic projection depends only
on the Laplacian structure and the quadratic scaling, since the equation
controls the trace of each component Hessian, whereas the projection records
its trace-free quadratic part.  By contrast, the uniform coefficient estimate
relies on the specific identities
\begin{equation*}
	\cN(\rho z)=\cN(z)\quad(\rho>0),
	\quad
	\cN(-z)=-\cN(z),
\quad\text{and}\quad
	z\cdot\cN(z)=|z|,
\end{equation*}
together with the regularity of $\cN$ away from the origin.  These properties
provide the cancellations in the affine-dominant range and the positive lower
bound required in the quadratic-dominant range, so the proof uses
substantially more than the estimate $|\cN|\le1$.

This distinction also clarifies the relation between the present result and
other questions of optimal second-order regularity.  For constraint maps, the
survey asks which regularity of the target boundary guarantees
$W^{2,\infty}_{\mathrm{loc}}$ regularity on the continuity set
\cite[Problem~7.1]{FGKS26} (see also
\cite{FKS24,FGKS24,FGKS25}).  In a different direction, Koike
\cite[Open questions~2 and~6]{Koi21} posed local $W^{2,\infty}$ questions for
scalar obstacle equations governed by Isaacs-type operators.

Nevertheless, the analogy lies in the regularity question rather than in the
analytic structure of the equations.  In particular, the argument developed
here does not apply directly to the problems described above, since the
uniform coefficient estimate relies essentially on the special identities
satisfied by $\cN$.  Accordingly, adapting this approach appears at present
to require additional estimates tailored to the structure of each equation.

\subsection{Proof strategy and organization}

For the reader's convenience, we now describe the main ideas in the proof of
Theorem~\ref{thm:main}, while indicating how the principal estimates fit
together.  
After translation and quadratic rescaling, it is enough to work
with a solution $\bv$ in $B_2$ at a point where $\bv(0)\ne0$.
On the logarithmic scale 
$$t=-\log r,$$
set
$\bA(t,\omega):=e^{2t}\bv(e^{-t}\omega)$, and let $\qtwo(t)$ denote the
coefficient obtained by projecting $\bA(t,\cdot)$ onto the vector-valued
harmonic quadratics on the sphere.  The localized Newtonian potential
decomposition separates $\bA$ into its affine part, this quadratic projection,
and a uniformly bounded remainder with zero quadratic projection
(Lemma~\ref{lem:decomposition}).  If $\Ftwo(t)$ denotes the corresponding
quadratic coefficient of $\cN(\bA(t,\cdot))$, projection of the equation in
polar coordinates gives the second-order ordinary differential system
\begin{equation*}
	\qtwo''(t)-(n+2)\qtwo'(t)=\Ftwo(t).
\end{equation*}
Noticing that $\lvert\qtwo(t)\rvert=O(e^{2t})$ and $n+2>2$, the exponential growth mode
$e^{(n+2)t}$ is excluded, and solving for $\qtwo'(t)$ therefore yields
\begin{equation*}
	\qtwo'(t)
	=-\int_t^\infty e^{-(n+2)(s-t)}\Ftwo(s)\,ds.
\end{equation*}
Although the bound $\lvert\cN\rvert\le1$ makes $\qtwo(t)$ to be Lipschitz, it does not
prevent its variation from accumulating over an unbounded interval of
logarithmic scales.  To rule out such cumulative growth, the proof must
therefore use more than the boundedness of $\cN$, and it is precisely here
that the normalized form of the nonlinearity becomes essential.

The resulting argument rests on two complementary estimates, each of which
is adapted to a different dominant configuration.  The first concerns a
scale $t_0$ at which the affine part dominates both the quadratic term and
the bounded remainder.  In this range, Lemma~\ref{lem:remaining-change}
shows that the projected right-hand side is integrable on all subsequent
scales, namely,
\begin{equation*}
	\int_{t_0}^{\infty}\lvert\Ftwo(t)\rvert\,dt\le C_n,
\end{equation*}
so that the integral representation controls the entire remaining variation
of $\qtwo(t)$.  The second estimate concerns scales at which the
quadratic coefficient is sufficiently large compared with the remainder
bound, whereas the affine contribution is small relative to that coefficient.
In this range, the estimate on the sphere gives a
uniform positive lower bound for the scalar product of $\qtwo(t)$ with $\Ftwo(t)$
at nearby later scales.  Lemma~\ref{lem:descent} combines this lower bound with
the integral representation for $\qtwo'$ and obtains the strict decrease
\begin{equation*}
	\frac{d}{dt}\lvert\qtwo(t)\rvert\le-c_{0,n}<0.
\end{equation*}
Thus, the first estimate prevents further accumulation once the affine part
has become dominant, whereas the second rules out the persistence of a large
quadratic coefficient before that transition.  The argument does
not require one of these two dominance conditions to hold at every scale,
since neither condition may apply when the affine and quadratic contributions
are comparable.  Proposition~\ref{prop:nonzero-center} bridges this
intermediate range by selecting a crossing scale at which the affine size
balances the quadratic size together with the remainder bound.  The
integrability estimate then controls all later scales, while the decrease
estimate rules out a large quadratic coefficient on a fixed interval before
the crossing.  Their combination produces uniform coefficient bounds at every
nonzero center.

Finally, the limiting quadratic coefficient is one half of the trace-free
part of the Hessian, whereas the equation controls its trace.  On the 
zero set, a two-fold application of the Sobolev level-set result shows that the
Hessian vanishes almost everywhere (Lemma~\ref{lem:sobolev-level-set}), after
which rescaling completes the proof.

The rest of the paper follows the preceding strategy.
Section~\ref{sec:matrix-expansion} introduces the matrix coefficients, proves the
Newtonian potential decomposition, and derives the scale equation, while
Section~\ref{sec:sphere-estimates} establishes the estimates on the sphere.
Section~\ref{sec:coefficient-bounds} then proves the uniform coefficient
bound, and Section~\ref{sec:final-proof} completes the proof of
Theorem~\ref{thm:main} by rescaling and treating the zero set.

\section{Matrix coefficients on the sphere and a bounded remainder}
\label{sec:matrix-expansion}

This section introduces the coefficients that encode the trace-free quadratic
information and develops the scale decomposition and differential identity
used in Sections~\ref{sec:sphere-estimates}
and~\ref{sec:coefficient-bounds}.

\subsection{Notation and the quadratic projection}

We denote the unit sphere in $\R^n$, centered at the origin, by
\begin{equation*}
	\Sn:=\partial B_1
	=\{\omega\in\R^n:|\omega|=1\}.
\end{equation*}
For a real matrix $M=(M_{ij})$, its Frobenius norm is
\begin{equation*}
	\|M\|_F=\left(\sum_{i,j}|M_{ij}|^2\right)^{1/2},
\end{equation*}
which is the Euclidean norm of the list of its entries.
We denote by $I_n$ the $n\times n$ identity matrix.

Let $\mathcal H^{n-1}$ be $(n-1)$-dimensional Hausdorff measure and define
the normalized surface measure on $\Sn$ by
\begin{equation*}
	d\sigma(\omega)
	:=\frac{d\mathcal H^{n-1}(\omega)}{\mathcal H^{n-1}(\Sn)}.
\end{equation*}
With this normalization, $\sigma(\Sn)=1$, and, unless stated otherwise, every
integral and every $L^p(\Sn;\R^m)$ norm over $\Sn$ is taken with respect to
$d\sigma$.  In particular,
$\|G\|_{L^p(\Sn)}=(\int_{\Sn}|G|^p\,d\sigma)^{1/p}$ for
$1\le p<\infty$, while $\|G\|_{L^\infty(\Sn)}$ denotes the essential supremum.

Define
\[
\mathcal X_{n,m}
:=\left\{H=(H^1,\ldots,H^m):
(H^\alpha)^T=H^\alpha,\ \operatorname{tr}H^\alpha=0
\text{ for }\alpha=1,\ldots,m\right\},
\]
where the superscript $T$ denotes transposition and $\operatorname{tr}$ denotes the sum of
the diagonal entries, so $\mathcal X_{n,m}$ equivalently consists of ordered
$m$-tuples of symmetric trace-free $n\times n$ matrices.
For $H,K\in\mathcal X_{n,m}$, we use the inner product and its
associated norm
\begin{align*}
\langle H,K\rangle
=\sum_{\alpha=1}^m\sum_{i,j=1}^n H_{ij}^\alpha K_{ij}^\alpha
=\sum_{\alpha=1}^m\operatorname{tr}(H^\alpha K^\alpha),
\end{align*}
and
\begin{align*}
|H|^2=\sum_{\alpha=1}^m\|H^\alpha\|_F^2
=\sum_{\alpha=1}^m\sum_{i,j=1}^n|H_{ij}^\alpha|^2,
\end{align*}
where the equality involving the trace uses the symmetry of the matrices,
and $\|H^\alpha\|_F$ is the Frobenius norm defined above.
For $\omega\in\Sn$, put
\begin{equation*}
	Q_H(\omega):=\bigl(\omega^TH^1\omega,\ldots,
	\omega^TH^m\omega\bigr).
\end{equation*}

Write $\omega=(\omega_1,\ldots,\omega_n)$, let $\delta_{ij}$ denote the
Kronecker symbol, and set
\begin{equation*}
	\kappa_n:=\frac{2}{n(n+2)}.
\end{equation*}
Here $\delta_{ij}$ equals $1$ when $i=j$ and $0$ otherwise.  For
$G\in L^1(\Sn;\R^m)$, define its quadratic matrix coefficient
$[G]_2\in\mathcal X_{n,m}$ componentwise by
\begin{equation}\label{eq:matrix-coefficient}
	\bigl([G]_2^\alpha\bigr)_{ij}
	=\kappa_n^{-1}\int_{\Sn}G_\alpha(\omega)
	\left(\omega_i\omega_j-\frac{\delta_{ij}}n\right)
	\,d\sigma(\omega),
	\qquad 1\le i,j\le n.
\end{equation}
The matrices in \eqref{eq:matrix-coefficient} are symmetric, while the
identity
\[
\sum_{i=1}^n\left(\omega_i^2-\frac1n\right)=0
\]
shows that they are trace free, and hence $[G]_2\in\mathcal X_{n,m}$.
When $G\in L^2(\Sn;\R^m)$, the normalization by $\kappa_n$ ensures that
$Q_{[G]_2}$ is the orthogonal projection of $G$ onto the space of
vector-valued trace-free quadratic functions.  The identities below verify
this interpretation and provide the estimates used later.

We first record two identities on the sphere
that underlie the projection formula.  Rotation invariance and
$\sum_{i=1}^n\omega_i^2=1$ imply
\begin{align}
	\int_{\Sn}\omega_i\omega_j\,d\sigma
	&=\frac{\delta_{ij}}n, \qquad i,j=1,\cdots,n
	\label{eq:sphere-coordinate-two}
\end{align}
and
\begin{align}
	\int_{\Sn}\omega_i\omega_j\omega_k\omega_l\,d\sigma
	&=\frac{\delta_{ij}\delta_{kl}+\delta_{ik}\delta_{jl}
		+\delta_{il}\delta_{jk}}{n(n+2)}, \qquad i,j,k,l=1,\cdots,n.
	\label{eq:sphere-coordinate-four}
\end{align}
For the first identity, reflection in a coordinate hyperplane makes the
integrals with $i\ne j$ vanish, while rotation invariance makes those with
$i=j$ equal, and their sum is
$\int_{\Sn}|\omega|^2\,d\sigma=1$.
To verify the second identity, every
integral in which some coordinate occurs to an odd power vanishes by the
same reflection argument.  By rotations, it remains to calculate
\[
a_n:=\int_{\Sn}\omega_1^4\,d\sigma,
	\quad\text{and}\quad
b_n:=\int_{\Sn}\omega_1^2\omega_2^2\,d\sigma.
\]
By rotational invariance, replacing $\omega_1$ by
$(\omega_1+\omega_2)/\sqrt2$ gives
\[
a_n=\frac14\int_{\Sn}(\omega_1+\omega_2)^4\,d\sigma
=\frac14(2a_n+6b_n),
\]
and hence $a_n=3b_n$.  On the other hand, expanding
$(\sum_{i=1}^n\omega_i^2)^2=1$ and integrating, we obtain
\[
1=na_n+n(n-1)b_n=n(n+2)b_n.
\]
Thus $b_n=1/[n(n+2)]$ and $a_n=3/[n(n+2)]$, which proves the stated formulas \eqref{eq:sphere-coordinate-two} and \eqref{eq:sphere-coordinate-four}.

Consequently, for $H,K\in\mathcal X_{n,m}$, one has
\begin{equation}\label{eq:Q-inner-product}
	\int_{\Sn}Q_H\cdot Q_K\,d\sigma
	=\kappa_n\langle H,K\rangle,
	\quad\text{and}\quad \|Q_H\|_{L^\infty(\Sn)}\le |H|.
\end{equation}
Expanding the integral and applying \eqref{eq:sphere-coordinate-four}, we find
for each component that
\[
\int_{\Sn}(\omega^TH^\alpha\omega)
(\omega^TK^\alpha\omega)\,d\sigma
=\frac{\operatorname{tr}H^\alpha\operatorname{tr}K^\alpha
	+2\operatorname{tr}(H^\alpha K^\alpha)}{n(n+2)},
\]
where the product
$\operatorname{tr}H^\alpha\operatorname{tr}K^\alpha$ vanishes in view of
$\operatorname{tr}H^\alpha=\operatorname{tr}K^\alpha=0$.  The pointwise
bound follows from the Cauchy--Schwarz inequality,
\[
|\omega^TH^\alpha\omega|
\le\left(\sum_{i,j}|H_{ij}^\alpha|^2\right)^{1/2}
\left(\sum_{i,j}\omega_i^2\omega_j^2\right)^{1/2}
=\|H^\alpha\|_F,
\]
where the last equality uses
$\sum_{i,j}\omega_i^2\omega_j^2=|\omega|^4=1$.
Squaring the componentwise estimate and summing over the target components
yields the corresponding vector bound, with a constant that remains
independent of $m$, namely
\[
|Q_H(\omega)|^2
=\sum_{\alpha=1}^m|\omega^TH^\alpha\omega|^2
\le\sum_{\alpha=1}^m\|H^\alpha\|_F^2
=|H|^2.
\]

Combining \eqref{eq:matrix-coefficient} with
$\operatorname{tr}H^\alpha=0$, we obtain, for every
$H\in\mathcal X_{n,m}$,
\begin{align*}
	\langle[G]_2,H\rangle
	&=\kappa_n^{-1}\sum_{\alpha=1}^m\sum_{i,j=1}^n
	H_{ij}^\alpha\int_{\Sn}G_\alpha(\omega)
	\left(\omega_i\omega_j-\frac{\delta_{ij}}n\right)d\sigma(\omega)\\
	&=\kappa_n^{-1}\int_{\Sn}\sum_{\alpha=1}^mG_\alpha(\omega)
	\left(\omega^TH^\alpha\omega-\frac{\operatorname{tr}H^\alpha}{n}\right)
	d\sigma(\omega)\\
	&=\kappa_n^{-1}\int_{\Sn}G(\omega)\cdot Q_H(\omega)\,d\sigma(\omega),
\end{align*}
where $[G]_2$ denotes the quadratic matrix coefficient defined by \eqref{eq:matrix-coefficient}.
Now take $G=Q_K$, where $K\in\mathcal X_{n,m}$, and use
\eqref{eq:Q-inner-product}.  Then
\[
\langle[Q_K]_2,H\rangle
=\kappa_n^{-1}\int_{\Sn}Q_K\cdot Q_H\,d\sigma
=\langle K,H\rangle
\qquad(H\in\mathcal X_{n,m}).
\]
Thus $[Q_K]_2-K$ is orthogonal to every element of $\mathcal X_{n,m}$.
Taking $H=[Q_K]_2-K$ in the last identity shows that
$|[Q_K]_2-K|^2=0$, and hence $[Q_K]_2=K$.  We have proved
\begin{equation}\label{eq:coefficient-identities}
	\langle[G]_2,H\rangle
	=\kappa_n^{-1}\int_{\Sn}G\cdot Q_H\,d\sigma,
		\quad\text{and}\quad [Q_H]_2=H.
\end{equation}

If $G\in L^2(\Sn;\R^m)$, these identities also show that
$Q_{[G]_2}$ is the orthogonal projection of $G$ in $L^2(\Sn;\R^m)$ onto the
space of trace-free quadratic functions.  Indeed, for every
$H\in\mathcal X_{n,m}$,
\begin{align*}
	\int_{\Sn}\bigl(G-Q_{[G]_2}\bigr)\cdot Q_H\,d\sigma
	=\kappa_n\langle[G]_2,H\rangle
	-\kappa_n\langle[G]_2,H\rangle=0.
\end{align*}
This proves the orthogonal-projection interpretation stated above.

Recalling \eqref{eq:sphere-coordinate-two}, constants have zero coefficient, while
the coefficient of a linear function $B\omega$, with
$B\in\R^{m\times n}$, also vanishes by the oddness of its product with
$\omega_i\omega_j-\delta_{ij}/n$.  Together with
$[Q_H]_2=H$, these
facts show that $[\,\cdot\,]_2$ annihilates constant and linear functions and
recovers $H$ from every trace-free quadratic form $Q_H$.

To make the relation between this projection and the Hessian explicit, let
$M^1,\ldots,M^m$ be symmetric $n\times n$ matrices and suppose that
\[
G_\alpha(\omega)=\frac12\omega^TM^\alpha\omega,
\quad\text{for}\quad \alpha=1,\ldots,m.
\]
Since $|\omega|=1$, we have
\[
G_\alpha(\omega)
=\frac12\omega^T\left(M^\alpha-
\frac{\operatorname{tr}M^\alpha}{n}I_n\right)\omega
+\frac{\operatorname{tr}M^\alpha}{2n}.
\]
The last term is constant and has zero quadratic matrix coefficient, while
the matrix in parentheses is symmetric and has zero trace.  The
identity in \eqref{eq:coefficient-identities} therefore reads, componentwise,
\begin{equation}\label{eq:quadratic-Hessian-coefficient}
	[G]_2^\alpha
	=\frac12\left(M^\alpha-
	\frac{\operatorname{tr}M^\alpha}{n}I_n\right),
\quad\text{for}\quad \alpha=1,\ldots,m.
\end{equation}
Consequently, the equation supplies the trace of each component Hessian,
whereas twice the corresponding quadratic coefficient is precisely the
trace-free part.

The coefficient identities also yield the bounds
\begin{equation}\label{eq:coefficient-bounds-basic}
	|[G]_2|\le\kappa_n^{-1}\|G\|_{L^1(\Sn)},
		\quad\text{and}\quad
	|[G]_2|\le\kappa_n^{-1/2}\|G\|_{L^2(\Sn)}.
\end{equation}
Here the first inequality holds for $G\in L^1(\Sn;\R^m)$, while the second one 
holds for $G\in L^2(\Sn;\R^m)$.  For the first inequality, if $|H|=1$, then
\[
|\langle[G]_2,H\rangle|
\le\kappa_n^{-1}\int_{\Sn}|G(\omega)|\,|Q_H(\omega)|
\,d\sigma(\omega)
\le\kappa_n^{-1}\|G\|_{L^1(\Sn)}.
\]
Taking the supremum over such $H$ yields the first bound.  For the second
bound, setting $K=H$ in \eqref{eq:Q-inner-product}, we have
\[
\|Q_H\|_{L^2(\Sn)}^2
=\int_{\Sn}|Q_H|^2\,d\sigma
=\kappa_n|H|^2.
\]
Consequently, $\|Q_H\|_{L^2(\Sn)}=\kappa_n^{1/2}$ when $|H|=1$.
The Cauchy--Schwarz inequality now implies
\[
|\langle[G]_2,H\rangle|
\le\kappa_n^{-1}\|G\|_{L^2(\Sn)}\|Q_H\|_{L^2(\Sn)}
=\kappa_n^{-1/2}\|G\|_{L^2(\Sn)}.
\]
Hence, taking the supremum over $|H|=1$ yields the second bound.

\subsection{Scale profiles and a bounded remainder}

We now apply the preceding quadratic projection to the rescaled traces of a
solution.  The purpose of this subsection is to separate the affine and
trace-free quadratic parts from a remainder that remains uniformly bounded
over all small scales, derive the scale equation for the quadratic
coefficient, and identify its limit with the trace-free Hessian at the
center.

For a weak solution $\bv$ of \eqref{eq:system} in $B_2$, we introduce the
$\R^m$-valued scale profile on $[0,\infty)\times\Sn$ by
\begin{equation}\label{eq:A-definition}
	\bA(t,\omega)=e^{2t}\bv(e^{-t}\omega).
\end{equation}
Using the quadratic coefficient defined above, we then set
\begin{equation}\label{eq:q-F2-definitions}
	\qtwo(t)=[\bA(t,\cdot)]_2，
	\qquad\text{and}\qquad
	\Ftwo(t)=[\cN(\bA(t,\cdot))]_2.
\end{equation}

\begin{lemma}\label{lem:decomposition}
	Let $\bv$ be a weak solution of \eqref{eq:system} in $B_2\subset\R^n$ such
	that $\|\bv\|_{L^\infty(B_2)}\le M$ and $\bv(0)\ne0$.   Then, for
	$t\ge1$,
	\begin{equation}\label{eq:decomposition}
		\bA(t,\omega)=e^{2t}\bv(0)+e^tD\bv(0)\omega
		+Q_{\qtwo(t)}(\omega)+\mathbf R(t,\omega),
	\end{equation}
	where
	\begin{equation}\label{eq:R-properties}
		[\mathbf R(t,\cdot)]_2=0,
		\qquad
		\sup_{t\ge1}\|\mathbf R(t,\cdot)\|_{L^\infty(\Sn)}
		\le C_n(M+1).
	\end{equation}
	Moreover,
	\begin{equation}\label{eq:qprime}
		\qtwo'(t)=-\int_t^\infty e^{-(n+2)(s-t)}\Ftwo(s)\,ds,
		\quad\text{and}\quad
		|\qtwo'(t)|\le\frac{\kappa_n^{-1/2}}{n+2}.
	\end{equation}
	The matrix coefficient $\qtwo(t)$ has a finite limit $\qtwo_\infty=(\qtwo_\infty^1,\cdots,\qtwo_\infty^m)$ as
	$t\to\infty$, and its components satisfy
	\begin{equation}\label{eq:q-Hessian-relation-1}
		\qtwo_\infty^\alpha
		=\frac12\left(D^2v_\alpha(0)
		-\frac{\Delta v_\alpha(0)}nI_n\right),
	\quad\text{for}\quad \alpha=1,\ldots,m,
	\end{equation}
	and
	\begin{equation}\label{eq:q-Hessian-relation-2}
		|D^2\bv(0)|^2=4|\qtwo_\infty|^2+\frac1n.
	\end{equation}
\end{lemma}

\begin{remark}
	The identities
	\eqref{eq:q-Hessian-relation-1} and \eqref{eq:q-Hessian-relation-2} show that
	controlling $\qtwo_\infty$ suffices to control the Hessian, since the
	equation determines the trace of each component Hessian, whereas
	$2\qtwo_\infty^\alpha$ is its trace-free part.  Moreover, the constant
	and linear terms in \eqref{eq:decomposition} have zero matrix coefficient,
	$[\mathbf R(t,\cdot)]_2=0$, and $[Q_{\qtwo(t)}]_2=\qtwo(t)$, so
	$Q_{\qtwo(t)}$ is the only term in the decomposition that carries the
	quadratic information required to recover the Hessian.
\end{remark}

\begin{proof}
	For clarity, we divide the proof into six steps, which fall naturally into
	three main stages.  In Steps~1--3 below, we decompose $\bv$ into a localized
	Newtonian potential and a harmonic function, and then use their Taylor
	expansions to separate the trace-free quadratic term from a uniformly bounded
	remainder.  Steps~4 and~5 pass to logarithmic variables, project the
	polar-coordinate equation onto the vector-valued harmonic quadratics, and
	derive both the ordinary differential equations for $\qtwo(t)$ and its integral
	representation.  Finally, Step~6 uses the assumption $\bv(0)\ne0$ and interior
	Schauder estimates near the origin to identify the limiting coefficient with
	one half of the trace-free part of $D^2\bv(0)$.
	
	\emph{Step 1. The local Newtonian potential and the harmonic part.}
	Set $\mathbf f:=\cN(\bv)=(f_1,\ldots,f_m)$, so that
	$|\mathbf f|\le1$ and $\Delta\bv=\mathbf f$.  Choose a function 
	$\psi\in C_c^\infty(B_{7/4})$ with $0\le\psi\le1$ and
	$\psi=1$ on $B_{3/2}$.  Let $\Gamma_n$ be the fundamental solution of the
	Laplacian in $\R^n$, normalized by $\Delta\Gamma_n=\delta_0$, where
	$\delta_0$ denotes the Dirac mass at the origin, and set
	\[
	\bw(x)=\int_{\R^n}\Gamma_n(x-y)\psi(y)\mathbf f(y)\,dy,
	\qquad \text{and}\qquad\bh(x)=\bv(x)-\bw(x),
	\]
	where $\psi\mathbf f$ is extended by zero outside $B_2$.
	The kernel $\Gamma_n$ is a constant multiple of $\log|x|$ when $n=2$ and
	of $|x|^{2-n}$ when $n\ge3$.  If $x\in B_{3/2}$ and
	$y\in\operatorname{supp}\psi\subset B_{7/4}$, then $x-y\in B_{13/4}$.
	Since $|\psi\mathbf f|\le1$, the change of variables $z=x-y$ yields
	\[
	|\bw(x)|
	\le\int_{B_{7/4}}|\Gamma_n(x-y)|\,dy
	\le\int_{B_{13/4}}|\Gamma_n(z)|\,dz.
	\]
	For $n\ge3$, we have
	\[
	\int_{B_{13/4}}|\Gamma_n(z)|\,dz
	\le C_n\int_0^{13/4}\rho^{2-n}\rho^{n-1}\,d\rho
	=C_n\int_0^{13/4}\rho\,d\rho<\infty.
	\]
	For $n=2$, the corresponding radial integral is
	$C_2\int_0^{13/4}\rho|\log\rho|\,d\rho<\infty$.  Therefore
	\[
	\|\bw\|_{L^\infty(B_{3/2})}\le C_n.
	\]
	Since $\Delta\bw=\psi\mathbf f=\mathbf f$ in $B_{3/2}$,
	$\bh=\bv-\bw$ is harmonic there.  Interior estimates for harmonic functions
	therefore yield
	\begin{equation}\label{eq:h-third-bound}
		\|D^3\bh\|_{L^\infty(B_{3/4})}
		\le C_n(M+1).
	\end{equation}

	\emph{Step 2. Estimates for the second-order Taylor remainder of the
		Newtonian kernel.}
	Fix $0<r\le e^{-1}$ and $\omega\in\Sn$, and observe that
	$x-y\in B_{9/4}$ whenever $x\in B_{1/2}$ and
	$y\in\operatorname{supp}\psi\subset B_{7/4}$.  Since
	$D\Gamma_n\in L^1_{\mathrm{loc}}(\R^n)$, the truncated kernel
	$D\Gamma_n\chi_{B_{9/4}}$ belongs to $L^1(\R^n)$.  Continuity of
	translations in $L^1$, together with $|\psi\mathbf f|\le1$, therefore shows that
	\[
	D\bw(x)=\int_{\R^n}D\Gamma_n(x-y)\psi(y)\mathbf f(y)\,dy
	\]
	has a continuous representative in $B_{1/2}$.  Subtracting the value and
	linear part of $\bw$ at the origin then gives
	\begin{align*}
		&\bw(r\omega)-\bw(0)-rD\bw(0)\omega\\
		={}&\int_{\R^n}
		\bigl(\Gamma_n(r\omega-y)-\Gamma_n(-y)
		-rD\Gamma_n(-y)\cdot\omega\bigr)\psi(y)\mathbf f(y)\,dy.
	\end{align*}
	We divide the integral into two parts $|y|<2r$ and $|y|\ge2r$, respectively.
	
	To estimate the integral over $|y|<2r$, set $y=rz$ and note that, when
	$n\ge3$, the homogeneity of the fundamental solution gives
	$\Gamma_n(rx)=r^{2-n}\Gamma_n(x)$ and
	$D\Gamma_n(rx)=r^{1-n}D\Gamma_n(x)$ for $r>0$.  Therefore
	\begin{align*}
		&\Gamma_n(r\omega-rz)-\Gamma_n(-rz)
		-rD\Gamma_n(-rz)\cdot\omega\\
		={}&r^{2-n}\bigl(\Gamma_n(\omega-z)-\Gamma_n(-z)
		-D\Gamma_n(-z)\cdot\omega\bigr).
	\end{align*}
	If $n=2$, then
	\[
	\Gamma_2(x)=\frac1{2\pi}\log|x|,
	\quad\text{and}\quad D\Gamma_2(x)=\frac1{2\pi}\frac{x}{|x|^2}.
	\]
	Consequently,
	\begin{align*}
		&\Gamma_2(r\omega-rz)-\Gamma_2(-rz)
		-rD\Gamma_2(-rz)\cdot\omega\\
		={}&\frac1{2\pi}\bigl(\log r+\log|\omega-z|
		-\log r-\log|z|
		+\tfrac{z\cdot\omega}{|z|^2}\bigr)\\
		={}&\frac1{2\pi}\left(\log|\omega-z|-\log|z|
		+\frac{z\cdot\omega}{|z|^2}\right).
	\end{align*}
	Since $dy=r^n\,dz$, the homogeneity factor and the Jacobian combine to
	produce $r^2$ in both cases, while $|\psi\mathbf f|\le1$ shows that inserting
	the factor $\psi(y)\mathbf f(y)$ preserves the same estimate.  Hence
	\begin{align}
		\int_{|y|<2r}
		\left|\Gamma_n(r\omega-y)-\Gamma_n(-y)
		-rD\Gamma_n(-y)\cdot\omega\right|\,dy
		\le C_n r^2.\label{eq:near-integral}
	\end{align}
	The remaining integral over $|z|<2$ is uniformly bounded with respect to
	$\omega$, as follows from the estimates below.  When $n\ge3$, translation
	of the ball and polar coordinates yield
	\begin{align*}
		\int_{|z|<2}|\Gamma_n(\omega-z)|\,dz
		&\le C_n\int_{|\zeta|<3}|\zeta|^{2-n}\,d\zeta
		=C_n\int_0^3\rho\,d\rho<\infty,
			\end{align*}
and
\begin{align*}
		\int_{|z|<2}|\Gamma_n(-z)|\,dz
		&\le C_n\int_0^2\rho\,d\rho<\infty.
	\end{align*}
	For $n=2$, the corresponding radial integral 
	$C_2\int_0^3\rho|\log\rho|\,d\rho$ is finite.  In every dimension under
	consideration,
	\[
	\int_{|z|<2}|D\Gamma_n(-z)|\,dz
	\le C_n\int_0^2\rho^{1-n}\rho^{n-1}\,d\rho
	=2C_n<\infty.
	\]
	These estimates are uniform in $\omega\in\Sn$ and complete the verification
	of \eqref{eq:near-integral}.
	
	For $|y|\ge2r$, Taylor's formula at $-y$ yields
	\begin{equation*}
		\Gamma_n(r\omega-y)=\Gamma_n(-y)+rD\Gamma_n(-y)\cdot\omega
		+\frac{r^2}{2}D^2\Gamma_n(-y)[\omega,\omega]
		+\mathcal R_n(r,y,\omega),
	\end{equation*}
	where
	\[
	D^2\Gamma_n(-y)[\omega,\omega]
	=\sum_{i,j=1}^n\partial_{ij}\Gamma_n(-y)\omega_i\omega_j
	\]
	is the quadratic expression determined by the Hessian matrix.
	The integral form of the Taylor remainder is
	\[
	\mathcal R_n(r,y,\omega)
	=\frac{r^3}{2}\int_0^1(1-s)^2
	D^3\Gamma_n(-y+sr\omega)[\omega,\omega,\omega]\,ds,
	\]
	where the trilinear expression is defined by
\begin{equation}
D^3\Gamma_n(\xi)[\omega,\omega,\omega]
:=\sum_{i,j,k=1}^n
\partial_{ijk}\Gamma_n(\xi)\omega_i\omega_j\omega_k.
\end{equation}
	Since $|D^3\Gamma_n(\xi)|\le C_n|\xi|^{-n-1}$ and
	$|-y+sr\omega|\ge|y|/2$ for $0\le s\le1$, the integral remainder satisfies
	\[
	|\mathcal R_n(r,y,\omega)|\le C_n r^3|y|^{-n-1}.
	\]
	Consequently,
	\begin{equation}\label{eq:far-remainder}
		\int_{{\{2r\le|y|<7/4\}}}
		|\mathcal R_n(r,y,\omega)|\,dy
		\le C_n r^3\int_{2r}^{7/4}\rho^{-2}\,d\rho
		=C_n r^3\left(\frac1{2r}-\frac4{7}\right)
		\le C_n r^2.
	\end{equation}
	For each $\alpha=1,\ldots,m$, the identity
	$\Delta\Gamma_n(-y)=0$ for $y\ne0$ shows that the matrix
	\[
	\frac12\int_{\{2r\le|y|<7/4\}}
	D^2\Gamma_n(-y)\psi(y)f_\alpha(y)\,dy
	\]
	is symmetric and trace-free, so the ordered tuple of these $m$ matrices
	belongs to $\mathcal X_{n,m}$.
	
	\emph{Step 3. Separation of the quadratic matrix
		coefficient and construction of the bounded remainder.}
	Taylor's formula for the harmonic part and \eqref{eq:h-third-bound} yield
	\[
	\bh(r\omega)-\bh(0)-rD\bh(0)\omega
	=\frac{r^2}{2}D^2\bh(0)[\omega,\omega]
	+\frac{r^3}{2}\int_0^1(1-s)^2
	D^3\bh(sr\omega)[\omega,\omega,\omega]\,ds.
	\]
	In this vector identity, the $\alpha$-th component of
	$D^2\bh(0)[\omega,\omega]$ is
	$$\sum_{i,j=1}^n\partial_{ij}h_\alpha(0)\omega_i\omega_j.$$
	For each $\alpha=1,\ldots,m$, define the $\alpha$-th matrix in
	$\widetilde{\mathbf q}(r)=(\widetilde q^1(r),\ldots,
	\widetilde q^m(r))$ by
	\begin{equation*}
		\widetilde q^\alpha(r)
		=\frac12D^2h_\alpha(0)
		+\frac12\int_{\{2r\le|y|<7/4\}}
		D^2\Gamma_n(-y)\psi(y)f_\alpha(y)\,dy,
	\end{equation*}
	where $h_\alpha$ denotes the $\alpha$-th component of $\bh$.
	Both terms are symmetric matrices, while the identities
	$\Delta h_\alpha(0)=0$ and $\Delta\Gamma_n(-y)=0$ for $y\ne0$ show that
	their traces vanish.  Consequently,
	\(\widetilde{\mathbf q}(r)\in\mathcal X_{n,m}\), and we define
	\(E_r(\omega)\) through the identity
	\begin{equation}\label{eq:preliminary-expansion}
		\frac{\bv(r\omega)-\bv(0)-rD\bv(0)\omega}{r^2}
		=Q_{\widetilde{\mathbf q}(r)}(\omega)+E_r(\omega).
	\end{equation}
	Moreover, the remainder  $E_r$ has the explicit form
	\begin{align*}
		E_r(\omega)
		={}&\frac1{r^2}\int_{|y|<2r}
		\bigl(\Gamma_n(r\omega-y)-\Gamma_n(-y)
		-rD\Gamma_n(-y)\cdot\omega\bigr)\psi(y)\mathbf f(y)\,dy\\
		&+\frac1{r^2}\int_{\{2r\le|y|<7/4\}}
		\mathcal R_n(r,y,\omega)\psi(y)\mathbf f(y)\,dy\\
		&+\frac r2\int_0^1(1-s)^2
		D^3\bh(sr\omega)[\omega,\omega,\omega]\,ds,
		\end{align*}
and satisfies
\begin{equation}\label{eq:E-r-bound}
	\|E_r\|_{L^\infty(\Sn)}
	\le C_n(M+1).
\end{equation}
	The estimate for $E_r$ follows term by term, since
	\eqref{eq:near-integral} and \eqref{eq:far-remainder} bound the first two
	integrals after division by $r^2$, while \eqref{eq:h-third-bound} bounds the
	third-order harmonic remainder by $C_n(M+1)$.
	
	To pass from \eqref{eq:preliminary-expansion} to
	\eqref{eq:decomposition}, set $r=e^{-t}$ and observe that
	\[
	\frac{\bv(r\omega)-\bv(0)-rD\bv(0)\omega}{r^2}
	=\bA(t,\omega)-e^{2t}\bv(0)-e^tD\bv(0)\omega.
	\]
	Noting that
	$\bigl[e^{2t}\bv(0)\bigr]_2=0$ and
	$\bigl[e^tD\bv(0)\omega\bigr]_2=0$, we see that the constant and linear terms
	make no contribution after applying $[\,\cdot\,]_2$.
	Taking quadratic matrix coefficients of both sides of
	\eqref{eq:preliminary-expansion} and using
	\eqref{eq:q-F2-definitions} and
	$[Q_{\widetilde{\mathbf q}(r)}]_2=\widetilde{\mathbf q}(r)$, we obtain
	\[
	\qtwo(t)=\widetilde{\mathbf q}(r)+[E_r]_2.
	\]
	We now define
	\[
	\mathbf R(t,\omega)=E_r(\omega)-Q_{[E_r]_2}(\omega),
	\]
	which, by \eqref{eq:coefficient-identities}, satisfies
	\[
	[\mathbf R(t,\cdot)]_2
	=[E_r]_2-[Q_{[E_r]_2}]_2
	=[E_r]_2-[E_r]_2=0.
	\]
	The estimates \eqref{eq:Q-inner-product} and
	\eqref{eq:coefficient-bounds-basic} further imply
	\[
	\|Q_{[E_r]_2}\|_{L^\infty(\Sn)}\le|[E_r]_2|
	\le\kappa_n^{-1/2}\|E_r\|_{L^2(\Sn)}
	\le\kappa_n^{-1/2}\|E_r\|_{L^\infty(\Sn)}.
	\]
The triangle inequality, the definition of $\mathbf R$, and
\eqref{eq:E-r-bound} now give
\begin{align*}
	\|\mathbf R(t,\cdot)\|_{L^\infty(\Sn)}
	&\le \|E_r\|_{L^\infty(\Sn)}
	+\|Q_{[E_r]_2}\|_{L^\infty(\Sn)}\\
	&\le (1+\kappa_n^{-1/2})
	\|E_r\|_{L^\infty(\Sn)}\\
	&\le C_n(M+1).
\end{align*}
	This proves \eqref{eq:decomposition}--\eqref{eq:R-properties}.
	
	\emph{Step 4. The
		ordinary differential equations for the matrix coefficient.}
	For a $C^2$ function $\phi$ on $\Sn$, extend it away from the origin by
	keeping it constant on every ray
	\[
	\widetilde\phi(x)=\phi\left(\frac{x}{|x|}\right),
	\qquad x\ne0,
	\]
	and define
	\[
	\Delta_{\Sn}\phi(\omega):=\Delta\widetilde\phi(\omega),
	\qquad \omega\in\Sn,
	\]
	which is the Laplace--Beltrami operator on the unit sphere.  Equivalently,
	direct differentiation in polar coordinates yields, for every $C^2$ scalar
	function $g$,
	\begin{equation}\label{eq:polar-laplacian}
		\Delta g(r\omega)=\partial_{rr}g(r\omega)
		+\frac{n-1}{r}\partial_rg(r\omega)
		+\frac1{r^2}\Delta_{\Sn}\bigl(g(r\,\cdot)\bigr)(\omega).
	\end{equation}
	Here $g(r\,\cdot)$ is the function on $\Sn$ obtained by fixing $r$, and
	the derivatives with respect to $r$ are taken with $\omega$ fixed.  For the
	present weak solution, choose $p>n$ in \eqref{eq:finite-p}.  On each compact
	annulus about the origin, smooth approximation in $W^{2,p}$ and passage to
	the limit justify the same polar-coordinate identity in the sense of
	distributions.
	
	Since $\bv(r\omega)=r^2\bA(t,\omega)$ and $$t=-\log r，$$ the chain rule yields
	\begin{align*}
		\partial_{rr}\bv(r\omega)&=2\bA-3\bA_t+\bA_{tt},\\
		\frac{n-1}{r}\partial_r\bv(r\omega)&=(n-1)(2\bA-\bA_t),
	\end{align*}
and
\begin{align*}
\frac1{r^2}\Delta_{\Sn}\bigl(\bv(r\,\cdot)\bigr)(\omega)
=\Delta_{\Sn}\bA(t,\omega).
\end{align*}
	All terms involving $\bA$ on the right-hand side are evaluated at $(t,\omega)$.
	Substitution in \eqref{eq:polar-laplacian} therefore results in
	\begin{equation}\label{eq:cylinder}
		\bA_{tt}-(n+2)\bA_t+\Delta_{\Sn}\bA+2n\bA=\cN(\bA),
	\end{equation}
	in the sense of distributions on $(0,\infty)\times\Sn$.

	For $H\in\mathcal X_{n,m}$ and $\alpha=1,\ldots,m$, define the 
	quadratic polynomial
	\[
	p_H^\alpha(x):=\sum_{i,j=1}^nH_{ij}^\alpha x_ix_j.
	\]
	Since $H^\alpha$ is symmetric and trace free, direct differentiation shows
	that
	\[
	\Delta p_H^\alpha=2\sum_{i=1}^nH_{ii}^\alpha=0.
	\]
	If $x=r\omega$, then $x_i=r\omega_i$, and hence
	\[
	p_H^\alpha(r\omega)
	=\sum_{i,j=1}^nH_{ij}^\alpha(r\omega_i)(r\omega_j)
	=r^2(Q_H)_\alpha(\omega).
	\]
	It follows componentwise that
\begin{align*}
\partial_{rr}p_H^\alpha(r\omega)=2(Q_H)_\alpha(\omega),
\end{align*}
and
\begin{align*}
\frac{n-1}{r}\partial_rp_H^\alpha(r\omega)
=2(n-1)(Q_H)_\alpha(\omega).
\end{align*}
	Applying \eqref{eq:polar-laplacian} to $p_H^\alpha$ yields
	\[
	0=2(Q_H)_\alpha+2(n-1)(Q_H)_\alpha
	+\Delta_{\Sn}(Q_H)_\alpha.
	\]
	Consequently,
	\begin{equation}\label{eq:sphere-quadratic-eigenvalue}
		\Delta_{\Sn}Q_H=-2nQ_H.
	\end{equation}
	For twice continuously differentiable functions $\phi_1$ and $\phi_2$, the
	integration-by-parts identity on the sphere reads
	\[
	\int_{\Sn}\phi_1\,\Delta_{\Sn}\phi_2\,d\sigma
	=\int_{\Sn}\phi_2\,\Delta_{\Sn}\phi_1\,d\sigma.
	\]
	Let $\eta\in C_c^\infty((0,\infty))$ and
	$H\in\mathcal X_{n,m}$.  Test \eqref{eq:cylinder} against
	$\eta(t)Q_H(\omega)$.  For almost every fixed $t$, integration by parts on the
	sphere and \eqref{eq:sphere-quadratic-eigenvalue} make the spherical part
	vanish
	\begin{align*}
		&\int_{\Sn}\bigl(\Delta_{\Sn}\bA(t,\omega)+2n\bA(t,\omega)\bigr)
		\cdot Q_H(\omega)\,d\sigma(\omega)\\
		={}&\int_{\Sn}\bA(t,\omega)\cdot
		\bigl(\Delta_{\Sn}Q_H(\omega)+2nQ_H(\omega)\bigr)
		\,d\sigma(\omega)=0.
	\end{align*}
On the other hand, for the remaining terms, the first identity in
	\eqref{eq:coefficient-identities} gives
	\[
	\int_{\Sn}\bA(t,\omega)\cdot Q_H(\omega)\,d\sigma
	=\kappa_n\langle\qtwo(t),H\rangle,
	\]
	and
	\[ \int_{\Sn}\cN(\bA(t,\omega))\cdot Q_H(\omega)\,d\sigma
	=\kappa_n\langle\Ftwo(t),H\rangle.
	\]
	After integrating by parts twice in $t$ for the term involving $\qtwo''$
	and once for the term involving $\qtwo'$, we obtain the distributional
	testing identity
	\begin{align*}
		0=\kappa_n\int_0^\infty
		\bigl(\langle\qtwo(t),H\rangle\eta''(t)
		+(n+2)\langle\qtwo(t),H\rangle\eta'(t)-\langle\Ftwo(t),H\rangle\eta(t)\bigr)\,dt,
	\end{align*}
where the compact support of $\eta$ eliminates all boundary terms.  By the
	definition of distributional derivatives, the displayed integral means that
	the scalar distribution
	$\langle\qtwo''-(n+2)\qtwo'-\Ftwo,H\rangle$ vanishes.  Hence, we obtain
	\[
	\left\langle\qtwo''-(n+2)\qtwo'-\Ftwo,H\right\rangle=0,
	\]
	for every $H\in\mathcal X_{n,m}$.  Since $\mathcal X_{n,m}$ is finite
	dimensional, these scalar distributional
	identities, taken against a basis of $\mathcal X_{n,m}$, are equivalent to the
	matrix-valued identity
	\begin{equation}\label{eq:q-ODE}
		\qtwo''-(n+2)\qtwo'=\Ftwo.
	\end{equation}
	Moreover, \eqref{eq:q-ODE} implies that
	$\qtwo\in C^{1,1}_{\mathrm{loc}}$, since it may be rewritten as
	\[
	\bigl(\qtwo'-(n+2)\qtwo\bigr)'=\Ftwo.
	\]
	Since $\Ftwo\in L^\infty_{\mathrm{loc}}$, the quantity
	$\qtwo'-(n+2)\qtwo$ has a locally Lipschitz representative.  The function
	$\qtwo$ is continuous by
	\eqref{eq:A-definition}--\eqref{eq:q-F2-definitions}, so the identity
	$\qtwo'=(n+2)\qtwo+(\qtwo'-(n+2)\qtwo)$ first yields
	$\qtwo\in C^1_{\mathrm{loc}}$.  The right-hand side of this identity is then
	locally Lipschitz, and hence $\qtwo\in C^{1,1}_{\mathrm{loc}}$ entrywise.
	If $t\ge0$ and $\omega\in\Sn$, then
	$e^{-t}\omega\in\overline B_1\subset B_2$.  Consequently,
	\eqref{eq:A-definition} and $\|\bv\|_{L^\infty(B_2)}\le M$ imply
	\[
	|\bA(t,\omega)|
	=e^{2t}|\bv(e^{-t}\omega)|
	\le Me^{2t}.
	\]
	Together with \eqref{eq:coefficient-bounds-basic}, this implies
	\begin{equation}\label{eq:q-growth}
		|\qtwo(t)|\le\kappa_n^{-1/2}M e^{2t}.
	\end{equation}
	
	\emph{Step 5. Solution of the ordinary differential
		equations for the matrix coefficients.}
	Since $|\cN(\bA)|\le1$, $\sigma(\Sn)=1$, and
	\eqref{eq:coefficient-bounds-basic} holds,
	\[
	|\Ftwo(t)|
	\le\kappa_n^{-1/2}\|\cN(\bA(t,\cdot))\|_{L^2(\Sn)}
	\le\kappa_n^{-1/2}.
	\]

	With $\lambda:=n+2$, multiplying
	$\qtwo''-\lambda\qtwo'=\Ftwo$ by $e^{-\lambda t}$ gives
	\[
	\frac d{dt}\bigl(e^{-\lambda t}\qtwo'(t)\bigr)
	=e^{-\lambda t}\Ftwo(t).
	\]
	For any fixed $t_0>0$, the preceding bound ensures that
	$e^{-\lambda s}\Ftwo(s)$ is absolutely integrable on $(t_0,\infty)$, so
	$e^{-\lambda t}\qtwo'(t)$ has a finite limit as $t\to\infty$, which we denote
	by $C$.  Integrating from $t$ to infinity and multiplying by
	$e^{\lambda t}$ yields
	\[
	\qtwo'(t)=-\int_t^\infty e^{-\lambda(s-t)}\Ftwo(s)\,ds+Ce^{\lambda t},
	\]
	where $C\in\mathcal X_{n,m}$ is independent of $t$.  Since the integral has
	norm at most $\kappa_n^{-1/2}/\lambda$, the assumption $C\ne0$ would imply,
	after integration from $t_0$ to $t$, that
	\[
	\qtwo(t)=\qtwo(t_0)+\frac{C}{\lambda}
	\bigl(e^{\lambda t}-e^{\lambda t_0}\bigr)
	-\int_{t_0}^t\int_\tau^\infty
	e^{-\lambda(s-\tau)}\Ftwo(s)\,ds\,d\tau.
	\]
	The norm of the double integral is at most
	$\kappa_n^{-1/2}(t-t_0)/\lambda$.  Consequently,
	$e^{-\lambda t}\qtwo(t)\to C/\lambda\ne0$.  This contradicts
	\eqref{eq:q-growth}, since $\lambda=n+2>2$ implies
	\[
	e^{-\lambda t}|\qtwo(t)|
	\le\kappa_n^{-1/2}M e^{-nt}\longrightarrow0\quad\text{as }t\to\infty.
	\]
	Hence $C=0$, proving the first formula in \eqref{eq:qprime}.  Finally, since
	$|\Ftwo(s)|\le\kappa_n^{-1/2}$,
	\[
	|\qtwo'(t)|
	\le\kappa_n^{-1/2}\int_t^\infty e^{-(n+2)(s-t)}\,ds
	=\frac{\kappa_n^{-1/2}}{n+2},
	\]
	which is the second estimate in \eqref{eq:qprime}.
	
	\emph{Step 6. Identification of the limit and recovery
		of the Hessian at the origin.}
	Finally, $\bv(0)\ne0$ and continuity imply that
	$|\bv|\ge|\bv(0)|/2$ in some ball $B_\rho$.  Choosing $p>n$ in
	\eqref{eq:finite-p} and applying Sobolev embedding, we obtain
	$\bv\in C^{1,\alpha}(B_\rho)$, after reducing $\rho$ if necessary, where
	$\alpha=1-n/p$.  The map \eqref{eq:N} has bounded derivative on every
	closed set contained in $\{z:|z|>|\bv(0)|/4\}$.  After decreasing $\rho$
	once more if necessary, the mean value theorem and
	$\bv\in C^{0,\alpha}(B_\rho)$ therefore imply
	$\cN(\bv)\in C^{0,\alpha}(B_\rho)$.  The interior Schauder estimate for
	$\Delta\bv=\cN(\bv)$ then yields
	$\bv\in C^{2,\alpha}(B_{\rho/2})$.  The resulting Taylor expansion is uniform
	in $\omega$ and has the form
	\[
	\bA(t,\omega)=e^{2t}\bv(0)+e^tD\bv(0)\omega
	+\frac12D^2\bv(0)[\omega,\omega]+o(1),
	\]
	where $o(1)\to0$ in $L^\infty(\Sn;\R^m)$.  The estimate
	\eqref{eq:coefficient-bounds-basic}, together with $\sigma(\Sn)=1$, also
	shows that the matrix coefficient of $o(1)$ tends to zero, namely
	\[
	|[o(1)]_2|
	\le\kappa_n^{-1}\|o(1)\|_{L^1(\Sn)}
	\le\kappa_n^{-1}\|o(1)\|_{L^\infty(\Sn)}
	\longrightarrow0.
	\]
	The constant and linear terms are annihilated by the projection and
	therefore make no contribution.  Applying
	\eqref{eq:quadratic-Hessian-coefficient} to each component with
	$M^\alpha=D^2v_\alpha(0)$, and then passing to the limit in the resulting
	coefficient identity, we obtain
	\[
	\qtwo_\infty^\alpha
	=\frac12\left(D^2v_\alpha(0)
	-\frac{\Delta v_\alpha(0)}nI_n\right),
	\]
	which proves \eqref{eq:q-Hessian-relation-1}.
	Since $\Delta\bv(0)=\cN(\bv(0))$,
	\eqref{eq:q-Hessian-relation-1} is equivalent to
	\[
	D^2v_\alpha(0)
	=2\qtwo_\infty^\alpha+\frac{\cN_\alpha(\bv(0))}nI_n.
	\]
	The identity $\operatorname{tr}\qtwo_\infty^\alpha=0$ makes the two
	matrices on the right orthogonal in the Frobenius inner product, while
	$\|I_n\|_F^2=n$.  Hence
	\[
	\|D^2v_\alpha(0)\|_F^2
	=4\|\qtwo_\infty^\alpha\|_F^2
	+\frac{|\cN_\alpha(\bv(0))|^2}{n}.
	\]
	Summing over $\alpha$ and using $|\cN(\bv(0))|=1$ yields
	\[
	|D^2\bv(0)|^2=4|\qtwo_\infty|^2+\frac1n,
	\]
	as claimed.
\end{proof}

\section{Quadratic projection estimates on the sphere}\label{sec:sphere-estimates}

This section establishes three estimates for $[\cN(\bA)]_2$, the quadratic
coefficient of the right-hand side on the sphere.  The first controls the
effect of a bounded perturbation, while the second proves integrability over
logarithmic scales when the affine part is dominant.  The third shows that,
when a sufficiently large trace-free quadratic term is dominant, the
projected right-hand side has a definite positive component in its direction.
Together, these estimates provide the two forms of control required in
Section~\ref{sec:coefficient-bounds}, namely  bounded total variation in the
affine-dominant range and strict decrease of the quadratic coefficient in the
quadratic-dominant range.

\begin{lemma}\label{lem:normalization-perturbation}
	Let $c\in\R^m$, $B\in\R^{m\times n}$, and
	\[
	L(\omega)=c+B\omega,
	\qquad S=|c|+\|B\|_F>0.
	\]
	Then, for $0<\delta<1$,
	\begin{equation}\label{eq:small-sphere-set}
		\mathcal H^{n-1}\bigl(\{\omega\in\Sn:|L(\omega)|\le\delta S\}\bigr)
		\le C_n\delta^{1/2}.
	\end{equation}
	Consequently, for every $E\in L^\infty(\Sn;\R^m)$,
	\begin{equation}\label{eq:normalization-difference}
		\int_{\Sn}|\cN(L+E)-\cN(L)|\,d\sigma
		\le C_n\min\left\{1,
		\left(\frac{\|E\|_{L^\infty(\Sn)}}{S}\right)^{1/2}\right\}.
	\end{equation}
\end{lemma}

\begin{proof}
	Writing $b:=\|B\|_F$, we infer from the Cauchy--Schwarz inequality that, for
	every $\omega\in\Sn$,
	\begin{align*}
		|B\omega|^2
		=\sum_{\alpha=1}^m
		\left(\sum_{j=1}^n B_{\alpha j}\omega_j\right)^2
		\le\sum_{\alpha=1}^m
		\left(\sum_{j=1}^nB_{\alpha j}^2\right)
		\left(\sum_{j=1}^n\omega_j^2\right)
		=\|B\|_F^2=b^2.
	\end{align*}
	If $|c|\ge2b$, then
	\[
	|c+B\omega|\ge|c|-|B\omega|
	\ge|c|-b\ge\frac S3.
	\]
	Thus the set in \eqref{eq:small-sphere-set} is empty when
	$0<\delta<1/3$, whereas for $1/3\le\delta<1$ its measure is at most
	$\mathcal H^{n-1}(\Sn)\le C_n\delta^{1/2}$, which proves the desired
	estimate whenever $|c|\ge2b$.
	
On the other hand, we now assume that $|c|<2b$, which implies that
$b>0$ and $S<3b$, and let $e_1,\ldots,e_n$ denote the standard basis
of $\R^n$.  Since
	\[
	b^2=\|B\|_F^2=\sum_{j=1}^n|Be_j|^2,
	\]
	there is an index $j_0$ such that
	\[
	|Be_{j_0}|\ge\frac b{\sqrt n}.
	\]
	To reduce the vector inequality to a scalar one, choose
	\[
	\zeta:=\frac{Be_{j_0}}{|Be_{j_0}|},
	\quad\text{and}\quad \rho:=|B^T\zeta|.
	\]
	Since $|\zeta|=1$, the $j_0$-th coordinate of $B^T\zeta$ satisfies
	\[
	(B^T\zeta)_{j_0}=\zeta\cdot Be_{j_0}=|Be_{j_0}|.
	\]
	Consequently,
	\[
	\rho\ge |Be_{j_0}|\ge\frac b{\sqrt n}>0.
	\]
	Hence the quantities
	\[
	\xi:=\frac{B^T\zeta}{\rho},
	\quad\text{and}\quad \tau:=\frac{\zeta\cdot c}{\rho}
	\]
	are well defined with $|\xi|=1$, and obey
	\[
	\zeta\cdot(c+B\omega)=\rho(\tau+\xi\cdot\omega)
	\quad \text{for}\quad \omega\in\Sn.
	\]
	If $|c+B\omega|\le\delta S$, then
	\[
	|\tau+\xi\cdot\omega|
	=\frac{|\zeta\cdot(c+B\omega)|}{\rho}
	\le\frac{|c+B\omega|}{\rho}
	\le\frac{\delta S}{\rho}<3\sqrt n\,\delta.
	\]
	Therefore the set in \eqref{eq:small-sphere-set} is contained in
	\begin{equation}\label{eq:coordinate-band}
		E_{\tau,\delta}:=
		\{\omega\in\Sn:|\tau+\xi\cdot\omega|\le3\sqrt n\,\delta\}.
	\end{equation}

	Since orthogonal transformations preserve $\mathcal H^{n-1}$ on $\Sn$,
	we may rotate the domain coordinates, assume that
	$\xi\cdot\omega=\omega_1$, and write
	\[
	\omega=(\cos\theta,\sin\theta\,\eta),
	\quad\text{for}\quad 0\le\theta\le\pi,\quad
	\eta\in\mathbb S^{n-2},
	\]
	where $\mathbb S^{n-2}$ is the unit sphere in $\R^{n-1}$.  In particular, for $n=2$,
	this formula is understood with $\mathbb S^0=\{-1,1\}$ and
	$\mathcal H^0(\mathbb S^0)=2$.

	Noting that the condition defining the band is
	$-3\sqrt n\,\delta\le\tau+\cos\theta\le3\sqrt n\,\delta$ and the map
	$\theta\mapsto\cos\theta$ is continuous and strictly decreasing on
	$[0,\pi]$, the admissible values of $\theta$ form either the empty set or a
	possibly degenerate closed interval $[\theta_1,\theta_2]\subset[0,\pi]$.
	In the nonempty case, the band therefore admits the parametrization
	\[
	E_{\tau,\delta}
	=\{(\cos\theta,\sin\theta\,\eta):\theta_1\le\theta\le\theta_2,
	\ \eta\in\mathbb S^{n-2}\}.
	\]
	In these coordinates, the surface element is
	$(\sin\theta)^{n-2}\,d\mathcal H^{n-2}(\eta)\,d\theta$, and hence the
	preceding parametrization yields
	\begin{align*}
		\mathcal H^{n-1}(E_{\tau,\delta})
		&=\int_{\theta_1}^{\theta_2}\int_{\mathbb S^{n-2}}
		(\sin\theta)^{n-2}\,d\mathcal H^{n-2}(\eta)\,d\theta\\
		&=\mathcal H^{n-2}(\mathbb S^{n-2})
		\int_{\theta_1}^{\theta_2}(\sin\theta)^{n-2}\,d\theta\\
		&\le \mathcal H^{n-2}(\mathbb S^{n-2})
		(\theta_2-\theta_1).
	\end{align*}
The possible values of $\omega_1$ in the band lie in an interval whose length
is bounded by $6\sqrt n\,\delta$, which implies that
	\[
	\cos\theta_1-\cos\theta_2\le6\sqrt n\,\delta.
	\]
	On the other hand,
	\begin{align*}
		\cos\theta_1-\cos\theta_2
		&=2\sin\left(\frac{\theta_1+\theta_2}{2}\right)
		\sin\left(\frac{\theta_2-\theta_1}{2}\right)\\
		&\ge2\sin^2\left(\frac{\theta_2-\theta_1}{2}\right)\\
		&\ge\frac{2}{\pi^2}(\theta_2-\theta_1)^2.
	\end{align*}
	Consequently,
	\[
	\theta_2-\theta_1
	\le\pi(3\sqrt n\,\delta)^{1/2}
	\le C_n\delta^{1/2},
	\]
	and therefore the band in \eqref{eq:coordinate-band} has measure at most
	$C_n\delta^{1/2}$ whenever it is nonempty.  Noting that the same conclusion is
	immediate for the empty band, this proves \eqref{eq:small-sphere-set}.

	For every $x,y\in\R^m$, the pointwise estimate
	\begin{equation}\label{eq:N-pointwise-difference}
		|\cN(x+y)-\cN(x)|
		\le4\min\{1,|y|/|x|\}
	\end{equation}
	holds, where the minimum is interpreted as $1$ when $x=0$.  Since
	$\cN(0)=0$ and $|\cN|\le1$, the estimate is immediate in this exceptional
	case, and we may henceforth assume that $x\ne0$.  If
	$|y|\ge|x|/2$, then
	\[
	|\cN(x+y)-\cN(x)|
	\le2\le4\min\left\{1,\frac{|y|}{|x|}\right\}.
	\]
	If $|y|<|x|/2$, every point $x+sy$, with $0\le s\le1$, on the segment
	joining $x$ to $x+y$ satisfies
	\[
	|x+sy|\ge |x|-s|y|>\frac{|x|}{2}.
	\]
	Writing $\widehat z:=z/|z|$, we compute from \eqref{eq:N} that, for
	$z\ne0$ and $h\in\R^m$,
	\[
	D\cN(z)h
	=\frac{h}{|z|}-\frac{(z\cdot h)z}{|z|^3}
	=\frac1{|z|}
	\bigl(h-(\widehat z\cdot h)\widehat z\bigr).
	\]
	Since the vector in parentheses is the orthogonal projection of $h$ onto
	$\widehat z^\perp$, it satisfies
	\[
	\left|h-(\widehat z\cdot h)\widehat z\right|^2
	=|h|^2-(\widehat z\cdot h)^2\le|h|^2.
	\]
	Consequently, for every $h\in\R^m$,
	\begin{equation}\label{eq:N-derivative-bound}
		|D\cN(z)h|\le\frac{|h|}{|z|},
		\qquad z\ne0.
	\end{equation}
	Applying the fundamental theorem of calculus along the segment and using
	\eqref{eq:N-derivative-bound}, we obtain
	\begin{align*}
		|\cN(x+y)-\cN(x)|
		\le\int_0^1|D\cN(x+sy)y|\,ds
		\le\int_0^1\frac{|y|}{|x+sy|}\,ds
		\le2\frac{|y|}{|x|}.
	\end{align*}
This, together with the preceding case, gives the estimate
	\eqref{eq:N-pointwise-difference}.

	Set $\varepsilon:=\|E\|_{L^\infty(\Sn)}/S$ and observe that the desired
	estimate is immediate when $\varepsilon=0$ or $\varepsilon\ge1$: in the
	first case, $\cN(L+E)=\cN(L)$ almost everywhere, while in the second case
	the integrand is bounded by $2$.  We may therefore assume that
	$0<\varepsilon<1$ and choose the unique integer $J\ge0$ such that
	\begin{equation}\label{eq:dyadic-J}
		2^J\varepsilon<1\le2^{J+1}\varepsilon.
	\end{equation}
	With this choice of $J$, decompose the sphere into the disjoint sets
	\begin{align*}
		\mathcal A_{-1}&=\{|L|\le\varepsilon S\},\\
		\mathcal A_j&=\{2^j\varepsilon S<|L|\le2^{j+1}\varepsilon S\},
		\qquad 0\le j\le J-1,
	\end{align*}
and
	\begin{align*}
		\mathcal A_J&=\{|L|>2^J\varepsilon S\},
	\end{align*}
	where the second family is omitted when $J=0$.  On $\mathcal A_{-1}$, the
	bound $2$ for the integrand and \eqref{eq:small-sphere-set} imply
	\[
	\int_{\mathcal A_{-1}}|\cN(L+E)-\cN(L)|\,d\sigma
	\le\frac{2\mathcal H^{n-1}(\mathcal A_{-1})}
	{\mathcal H^{n-1}(\Sn)}
	\le C_n\varepsilon^{1/2}.
	\]
	On $\mathcal A_j$, inequality
	\eqref{eq:N-pointwise-difference}, with $x=L(\omega)$ and
	$y=E(\omega)$, implies the pointwise bound $2^{2-j}$.  Since
	$2^{j+1}\varepsilon<1$ for $j\le J-1$, the set $\mathcal A_j$ is contained
	in $\{|L|\le2^{j+1}\varepsilon S\}$, whose measure is bounded by
	$C_n(2^{j+1}\varepsilon)^{1/2}$.  Consequently,
	\begin{align*}
		\sum_{j=0}^{J-1}\int_{\mathcal A_j}
		|\cN(L+E)-\cN(L)|\,d\sigma
		&\le C_n\sum_{j=0}^{J-1}
		2^{-j}(2^{j+1}\varepsilon)^{1/2}\\
		&\le C_n\varepsilon^{1/2}
		\sum_{j=0}^\infty2^{-j/2}
		\le C_n\varepsilon^{1/2}.
	\end{align*}
	Finally, on $\mathcal A_J$, inequality
	\eqref{eq:N-pointwise-difference} gives the pointwise bound
	$2^{2-J}$, while \eqref{eq:dyadic-J} implies
	$2^{-J} \le 2\varepsilon$.  The contribution of this set is therefore bounded by $8\varepsilon^{1/2}$, which, together with the two
	preceding estimates, proves \eqref{eq:normalization-difference}.
\end{proof}

The next lemma integrates the preceding comparison estimate over the
logarithmic-scale variable $t$, while allowing either the constant term or
the linear term to vanish.

\begin{lemma}
	\label{lem:weighted-sphere-integrability}
	For $a\in\R^m$ and $P\in\R^{m\times n}$, set
	\[
	L(t,\omega)=e^{2t}a+e^tP\omega.
	\]
	Then there exists a positive  constant $C_n$, depending only on $n$, such that
	\begin{equation}\label{eq:weighted-sphere-L1}
		\int_{-\infty}^{\infty}
		\left|[\cN(L(t,\cdot))]_2\right|\,dt\le C_n.
	\end{equation}
\end{lemma}

\begin{proof}
	If $a=0$, the invariance of $\cN$ under positive rescaling implies that
	$\cN(L(t,\omega))=\cN(P\omega)$, while its oddness gives
	\[
	\cN(P(-\omega))=-\cN(P\omega).
	\]
	This identity holds on the whole sphere, including the points at which
	$P\omega=0$, where both sides vanish. The function $\omega\mapsto\cN_\alpha(L(t,\omega))$ is odd, whereas
	$\omega\mapsto\omega_i\omega_j-\delta_{ij}/n$ is even.  Their product is
	therefore odd.  Since the antipodal map $\omega\mapsto-\omega$ preserves the
	surface measure on $\Sn$, we have
	\begin{align*}
		&\int_{\Sn}\cN_\alpha(L(t,\omega))
		\left(\omega_i\omega_j-\frac{\delta_{ij}}n\right)\,d\sigma(\omega)\\
		\qquad
		=&-\int_{\Sn}\cN_\alpha(L(t,\omega))
		\left(\omega_i\omega_j-\frac{\delta_{ij}}n\right)\,d\sigma(\omega)=0.
	\end{align*}
	Thus every entry in \eqref{eq:matrix-coefficient} vanishes, and consequently
	\[
	[\cN(L(t,\cdot))]_2=0.
	\]
	
For the trivial case $P=0$, then $\cN(L(t,\omega))$ is independent of $\omega$, while
	\eqref{eq:sphere-coordinate-two} implies that
	\[
	\int_{\Sn}
	\left(\omega_i\omega_j-\frac{\delta_{ij}}n\right)\,d\sigma(\omega)=0,
	\quad\text{for}\quad 1\le i,j\le n.
	\]
	and hence its quadratic coefficient vanishes as well.
	
	It remains to consider the non-trivial case $a\ne0$ and $P\ne0$, for which we set
	\[
	e=\frac a{|a|},\qquad
	\widehat P=\frac P{\|P\|_F},\quad \text{and}\quad
	\beta(t)=\frac{e^t|a|}{\|P\|_F}.
	\]
	The Cauchy--Schwarz inequality implies
	\begin{equation}\label{eq:P-hat-pointwise}
		|\widehat P\omega|\le\|\widehat P\|_F|\omega|=1,
		\quad\text{for}\quad \omega\in\Sn.
	\end{equation}
	Since $d\beta(t)/dt=\beta(t)$, we have $dt=d\beta/\beta$, while the identity
	$\cN(\rho z)=\cN(z)$ for $\rho>0$ allows us to write
	\[
	\cN(L(t,\omega))=\cN(\beta e+\widehat P\omega).
	\]
Noting that the function $\omega\mapsto\cN(\widehat P\omega)$ is odd on the whole
	sphere, including at points where $\widehat P\omega=0$, and
	$\omega_i\omega_j-\delta_{ij}/n$ is even,
	\eqref{eq:matrix-coefficient} implies that
	$[\cN(\widehat P\omega)]_2=0$.
	
	For $0<\beta\le1$, apply
	\eqref{eq:normalization-difference} with $c=0$, $B=\widehat P$, and
	$E=\beta e$.  In Lemma~\ref{lem:normalization-perturbation}, the
	corresponding quantity $S$ is
	$\|\widehat P\|_F=1$ and
	$\|E\|_{L^\infty(\Sn)}=\beta$, so
	\[
	\|\cN(\beta e+\widehat P\omega)-\cN(\widehat P\omega)\|_{L^1(\Sn)}
	\le C_n\beta^{1/2}.
	\]
	Linearity of the matrix coefficient and the first estimate in
	\eqref{eq:coefficient-bounds-basic} imply
	\begin{align}
		\left|[\cN(\beta e+\widehat P\omega)]_2\right|
		&=\left|[\cN(\beta e+\widehat P\omega)
		-\cN(\widehat P\omega)]_2\right|\notag\\
		&\le\kappa_n^{-1}
		\|\cN(\beta e+\widehat P\omega)
		-\cN(\widehat P\omega)\|_{L^1(\Sn)}\notag\\
		&\le C_n\beta^{1/2}. \label{eq:small-beta}
	\end{align}

	For $\beta\ge2$, set $z:=\widehat P\omega$, which satisfies $|z|\le1$ by
	\eqref{eq:P-hat-pointwise}.  Differentiating \eqref{eq:N} at the unit
	vector $e$, we obtain
	\[
	D\cN(e)h=h-(e\cdot h)e,
	\qquad h\in\R^m,
	\]
	where we used $|e|=1$.  Differentiating once more, we find that, for
	$\zeta\ne0$,
	\begin{align*}
		D^2\cN(\zeta)[h,k]
		={}&-|\zeta|^{-3}\bigl((\zeta\cdot k)h+(\zeta\cdot h)k
		+(h\cdot k)\zeta\bigr)\\
		&+3|\zeta|^{-5}(\zeta\cdot h)(\zeta\cdot k)\zeta.
	\end{align*}
	The Cauchy--Schwarz inequality therefore implies
	\[
	|D^2\cN(\zeta)[h,k]|\le6|\zeta|^{-2}|h|\,|k|.
	\]
	In particular, if $|\zeta-e|\le1/2$, then $|\zeta|\ge1/2$ and
	\begin{equation}\label{eq:N-second-derivative-bound}
		|D^2\cN(\zeta)[h,k]|\le24|h|\,|k|,
	\end{equation}
	where the constant $24$ is independent of the target dimension $m$.
	With this dimension-free estimate at hand, Taylor's formula with integral
	remainder reads
	\begin{align*}
		\cN(e+\beta^{-1}z)
		=\cN(e)+\beta^{-1}D\cN(e)z
		+\beta^{-2}\int_0^1(1-s)
		D^2\cN(e+s\beta^{-1}z)[z,z]\,ds.
	\end{align*}
	Denote the integral remainder in this expansion, which will be estimated
	uniformly in $\omega$, by
	\[
	\mathcal R_\beta(\omega)
	:=\beta^{-2}\int_0^1(1-s)
	D^2\cN(e+s\beta^{-1}\widehat P\omega)
	[\widehat P\omega,\widehat P\omega]\,ds.
	\]
	Combining \eqref{eq:P-hat-pointwise} with
	\eqref{eq:N-second-derivative-bound}, we obtain
	\[
	\|\mathcal R_\beta\|_{L^\infty(\Sn)}
	\le24\beta^{-2}\int_0^1(1-s)\,ds
	=12\beta^{-2}.
	\]
	Using $\cN(e)=e$, $D\cN(e)z=z-(e\cdot z)e$, and the positive scale
	invariance $\cN(\beta e+z)=\cN(e+\beta^{-1}z)$, we may now write
	\[
	\cN(\beta e+\widehat P\omega)
	=e+\beta^{-1}\bigl(\widehat P\omega
	-(e\cdot\widehat P\omega)e\bigr)
	+\mathcal R_\beta(\omega).
	\]
	Since the constant term is independent of $\omega$,
	\eqref{eq:sphere-coordinate-two} shows that each entry of its quadratic
	matrix coefficient vanishes.  Moreover, the linear term is odd under the
	transformation $\omega\mapsto-\omega$, whereas
	$\omega_i\omega_j-\delta_{ij}/n$ is even.  Its quadratic matrix coefficient
	therefore vanishes as well.  Consequently, applying $[\,\cdot\,]_2$ to the
	preceding identity leaves only the quadratic coefficient of the remainder
	\[
	[\cN(\beta e+\widehat P\omega)]_2=[\mathcal R_\beta]_2.
	\]
	The first estimate in \eqref{eq:coefficient-bounds-basic}, together with
	$\sigma(\Sn)=1$, implies that
	\[
	|[\mathcal R_\beta]_2|
	\le\kappa_n^{-1}\|\mathcal R_\beta\|_{L^1(\Sn)}
	\le\kappa_n^{-1}\|\mathcal R_\beta\|_{L^\infty(\Sn)}
	\le C_n\beta^{-2}.
	\]
	Consequently,
	\begin{equation}\label{eq:large-beta-projection}
		\left|[\cN(\beta e+\widehat P\omega)]_2\right|
		\le C_n\beta^{-2},\quad\text{for}\quad \beta\ge2.
	\end{equation}
	
	Next, we consider the case \(1\le\beta\le2\). The bounds $|\cN|\le1$, $\sigma(\Sn)=1$, and the
	second estimate in \eqref{eq:coefficient-bounds-basic} imply
	\[
	\left|[\cN(\beta e+\widehat P\omega)]_2\right|
	\le\kappa_n^{-1/2}.
	\]
	After the change of variables $dt=d\beta/\beta$, the estimates for
	$0<\beta\le1$, $1\le\beta\le2$, and $\beta\ge2$ show that the three
	corresponding contributions are controlled, respectively, by
	\[
	\int_0^1\beta^{1/2}\frac{d\beta}{\beta}=2,
	\qquad
	\int_1^2\frac{d\beta}{\beta}=\log2,
	\quad\text{and}\quad
	\int_2^\infty\beta^{-2}\frac{d\beta}{\beta}=\frac18.
	\]
	Since these quantities are finite and independent of $a$, $P$, and $m$,
	their sum leads to \eqref{eq:weighted-sphere-L1}.
\end{proof}

The final estimate treats the regime in which the quadratic term $Q_H$
dominates the perturbation $W$ by combining an elementary pointwise
inequality with integration over the sphere.

\begin{lemma}\label{lem:coercivity}
	Let $H\in\mathcal X_{n,m}$ and $W\in L^1(\Sn;\R^m)$.  Then
	\begin{equation}\label{eq:coercivity}
		\left\langle H,[\cN(Q_H+W)]_2\right\rangle
		\ge |H|-2\kappa_n^{-1}\|W\|_{L^1(\Sn)}.
	\end{equation}
\end{lemma}

\begin{proof}
	Since the conclusion is immediate when $H=0$, we assume throughout the
	proof that $H\ne0$.  For almost every $\omega\in\Sn$, the 
	inequality 
	\begin{equation}\label{eq:coercivity-pointwise}
		Q_H\cdot\cN(Q_H+W)
		=|Q_H+W|-W\cdot\cN(Q_H+W)
		\ge|Q_H|-2|W|
	\end{equation}
remains valid even when $Q_H(\omega)+W(\omega)=0$.
	Here the triangle inequality gives $|Q_H+W|\ge|Q_H|-|W|$, while
	$|W\cdot\cN(Q_H+W)|\le|W|$ follows from $|\cN|\le1$ and remains valid
	when $Q_H+W=0$ under the convention $\cN(0)=0$.
	
	To obtain the required lower bound for $\int_{\Sn}|Q_H|\,d\sigma$, observe
	first that
	\[
	|Q_H(\omega)|
	\le\|Q_H\|_{L^\infty(\Sn)}.
	\]
	Integrating and using \eqref{eq:Q-inner-product}, we obtain
	\[
	\kappa_n|H|^2
	=\|Q_H\|_{L^2(\Sn)}^2
	\le\|Q_H\|_{L^\infty(\Sn)}
	\int_{\Sn}|Q_H|\,d\sigma
	\le|H|\int_{\Sn}|Q_H|\,d\sigma.
	\]
	The final inequality follows from
	$\|Q_H\|_{L^\infty(\Sn)}\le|H|$.  Since $H\ne0$, we may divide by $|H|$
	to obtain
	\[
	\int_{\Sn}|Q_H|\,d\sigma\ge\kappa_n|H|.
	\]
	Finally, integrating \eqref{eq:coercivity-pointwise} with respect to
	$d\sigma$, multiplying by $\kappa_n^{-1}$, and applying the first identity
	in \eqref{eq:coefficient-identities} yields \eqref{eq:coercivity}, thereby
	completing the proof.
\end{proof}

\section{Uniform bounds for the quadratic coefficients}
\label{sec:coefficient-bounds}

The three estimates established in the preceding section, when combined with
the differential formula \eqref{eq:qprime}, yield two forms of control
corresponding to the relative contributions of the affine and quadratic
parts.  When the affine part dominates both the quadratic term and the
bounded remainder, the projected right-hand side is integrable over all
subsequent logarithmic scales, which bounds the remaining variation of
$\qtwo(t)$.  A second estimate applies when the quadratic coefficient is
sufficiently large relative to the remainder, while the affine part is
sufficiently small relative to that coefficient.  Under these assumptions,
$|\qtwo(t)|$ has a strictly negative derivative. Although the hypotheses of these two estimates do not exhaust all possible
configurations at each scale, the proposition at the end of this section
exploits the exponential growth of the affine part and the continuity of
$\qtwo(t)$ to select a crossing scale at which the affine contribution
balances the quadratic coefficient together with the remainder bound, thereby
connecting the two estimates and establishing a uniform bound for the
quadratic coefficients.

Fix a solution satisfying Lemma~\ref{lem:decomposition}, so that
$\bv(0)\ne0$, $\qtwo(t)\to\qtwo_\infty$, and
\eqref{eq:decomposition} holds, and set
\begin{equation*}
	L(t,\omega)=e^{2t}\bv(0)+e^tD\bv(0)\omega,
	\qquad
	S(t)=e^{2t}|\bv(0)|+e^t\|D\bv(0)\|_F.
\end{equation*}
The componentwise Cauchy--Schwarz inequality implies that
\[
|D\bv(0)\omega|\le\|D\bv(0)\|_F
\qquad(\omega\in\Sn),
\]
and therefore $|L(t,\omega)|\le S(t)$.

Writing $\lambda=n+2$ throughout this section, we infer from the bounds in
\eqref{eq:coefficient-bounds-basic}, the pointwise estimate
$|\cN(\bA)|\le1$, and $\sigma(\Sn)=1$ that
\begin{equation}\label{eq:F2-uniform-bound}
	|\Ftwo(t)|
	\le\kappa_n^{-1/2}\|\cN(\bA(t,\cdot))\|_{L^2(\Sn)}
	\le\kappa_n^{-1/2}.
\end{equation}
Combining \eqref{eq:F2-uniform-bound} with the representation
\eqref{eq:qprime}, we obtain
\begin{equation}\label{eq:q-Lipschitz}
	|\qtwo(t+\tau)-\qtwo(t)|
	\le\int_t^{t+\tau}|\qtwo'(s)|\,ds
	\le\frac{\kappa_n^{-1/2}}{n+2}\,\tau,
	\quad\text{for}\quad \tau\ge0.
\end{equation}
The definition of $S$ also gives
\begin{equation}\label{eq:S-growth}
	e^\tau S(t)\le S(t+\tau)\le e^{2\tau}S(t),
	\quad\text{for}\quad \tau\ge0.
\end{equation}
In particular, since $\bv(0)\ne0$,
\[
S'(t)=2e^{2t}|\bv(0)|+e^t\|D\bv(0)\|_F>0,
\]
so $S$ is strictly increasing.

We first treat a scale at which the affine part dominates the quadratic term
and the bounded remainder.  The next lemma shows that this condition makes
$\Ftwo(t)$ integrable on all later scales and hence bounds the remaining
variation of $\qtwo(t)$.

\begin{lemma}
	\label{lem:remaining-change}
	Suppose that there exists $K_R\ge1$ such that
	\begin{equation}\label{eq:R-uniform-bound}
		\sup_{t\ge1}\|\mathbf R(t,\cdot)\|_{L^\infty(\Sn)}\le K_R.
	\end{equation}
	Then there is a constant $C_n$ such that, if
	\begin{equation}\label{eq:large-S-assumption}
		S(t_0)\ge |\qtwo(t_0)|+K_R,
		\quad \text{for some}\quad t_0\ge1,
	\end{equation}
	then
	\begin{equation}\label{eq:remaining-change-conclusion}
		|\qtwo_\infty-\qtwo(t_0)|\le C_n.
	\end{equation}
\end{lemma}

\begin{proof}
	Recalling $E(t,\omega)=Q_{\qtwo(t)}(\omega)+\mathbf R(t,\omega)$, the identity
	\eqref{eq:Q-inner-product}, the Lipschitz
	estimate \eqref{eq:q-Lipschitz}, the growth bound \eqref{eq:S-growth}, and
	the hypothesis \eqref{eq:large-S-assumption} imply, for $t=t_0+\tau$, that
	\begin{align*}
		\frac{\|E(t,\cdot)\|_{L^\infty(\Sn)}}{S(t)}
		&\le\frac{|\qtwo(t)|+K_R}{S(t)}\\
		&\le\frac{|\qtwo(t_0)|+
			\dfrac{\kappa_n^{-1/2}}{n+2}\tau+K_R}
		{(|\qtwo(t_0)|+K_R)e^\tau}\\
		&\le C_n(1+\tau)e^{-\tau}.
	\end{align*}
	Since \eqref{eq:decomposition} gives $\bA=L+E$, the linearity of the matrix
	coefficient implies that
	\[
	\Ftwo(t)-[\cN(L(t,\cdot))]_2
	=[\cN(L(t,\cdot)+E(t,\cdot))-\cN(L(t,\cdot))]_2.
	\]
	Using the first estimate in \eqref{eq:coefficient-bounds-basic} to pass
	from the $L^1(\Sn)$ difference to its matrix coefficient, and then applying
	\eqref{eq:normalization-difference}, we find for $t=t_0+\tau$ that
	\begin{align}
		\left|\Ftwo(t)-[\cN(L(t,\cdot))]_2\right|
		&\le\kappa_n^{-1}
		\|\cN(L(t,\cdot)+E(t,\cdot))-\cN(L(t,\cdot))\|_{L^1(\Sn)}\notag\\
		&\le C_n\min\left\{1,
		\left(\frac{\|E(t,\cdot)\|_{L^\infty(\Sn)}}{S(t)}\right)^{1/2}\right\}\notag\\
		&\le C_n\min\{1,(1+\tau)^{1/2}e^{-\tau/2}\}. \label{eq:F2-linear-difference}
	\end{align}
	The right-hand side is integrable on $(0,\infty)$, since it is bounded on
	$(0,1)$, while
	$(1+\tau)^{1/2}e^{-\tau/2}\le(1+\tau)e^{-\tau/2}$ for $\tau\ge1$.
	Moreover, Lemma~\ref{lem:weighted-sphere-integrability} ensures that
	\[
	\int_{t_0}^\infty|[\cN(L(t,\cdot))]_2|\,dt
	\le\int_{-\infty}^\infty|[\cN(L(t,\cdot))]_2|\,dt
	\le C_n.
	\]
	Combining this bound with the integrable majorant in
	\eqref{eq:F2-linear-difference}, we obtain
	\begin{equation}\label{eq:F2-integrable-after-t0}
		\int_{t_0}^\infty|\Ftwo(t)|\,dt\le C_n.
	\end{equation}
	The representation \eqref{eq:qprime}, estimate
	\eqref{eq:F2-integrable-after-t0}, and Tonelli's theorem then give
	\begin{align*}
		\int_{t_0}^\infty|\qtwo'(t)|\,dt
		&\le\int_{t_0}^\infty\int_t^\infty
		e^{-\lambda(s-t)}|\Ftwo(s)|\,ds\,dt\\
		&=\int_{t_0}^\infty|\Ftwo(s)|
		\left(\int_{t_0}^s e^{-\lambda(s-t)}\,dt\right)ds\\
		&=\int_{t_0}^\infty|\Ftwo(s)|
		\frac{1-e^{-\lambda(s-t_0)}}{\lambda}\,ds\\
		&\le\frac1\lambda\int_{t_0}^\infty|\Ftwo(s)|\,ds
		\le C_n.
	\end{align*}
	Consequently,
	\[
	|\qtwo_\infty-\qtwo(t_0)|
	\le\int_{t_0}^\infty|\qtwo'(t)|\,dt\le C_n,
	\]
	which proves \eqref{eq:remaining-change-conclusion}.
\end{proof}

Whereas Lemma~\ref{lem:remaining-change} treats the range in which $S$ is
large, the following lemma supplies the complementary derivative estimate
when the quadratic coefficient dominates both $S$ and the bounded remainder.

\begin{lemma}
	\label{lem:descent}
	Suppose that $K_R\ge1$ and \eqref{eq:R-uniform-bound} holds.  Then there exist
	constants
	$\varepsilon_n,c_{0,n}>0$ and $C_n^*\ge1$ depending only on $n$, such that for any $t\ge1$, the inequalities 
	\begin{equation}\label{eq:decrease-assumptions}
		|\qtwo(t)|\ge(K_R+1)C_n^*,
		\qquad S(t)\le\varepsilon_n|\qtwo(t)|
	\end{equation}
	imply
	\begin{equation}\label{eq:descent}
		\frac d{dt}|\qtwo(t)|\le-c_{0,n}.
	\end{equation}
\end{lemma}

\begin{proof}
	We first choose $D_n>1$ sufficiently large that
	\begin{equation}\label{eq:D-choice}
		\kappa_n^{-1/2}e^{-\lambda D_n}
		\le\frac14(1-e^{-\lambda D_n}).
	\end{equation}
	Having fixed $D_n$, we choose
	\begin{equation}\label{eq:epsilon-choice}
		0<\varepsilon_n\le
		\frac{\kappa_n}{8e^{2D_n}}.
	\end{equation}
	For $0\le\tau\le D_n$, inequalities \eqref{eq:q-Lipschitz},
	\eqref{eq:S-growth}, and \eqref{eq:decrease-assumptions} yield
	\[
	|\qtwo(t+\tau)-\qtwo(t)|
	\le\frac{\kappa_n^{-1/2}}{n+2}D_n,
	\quad\text{and}\quad
	S(t+\tau)\le e^{2D_n}\varepsilon_n|\qtwo(t)|.
	\]
	Applying Lemma~\ref{lem:coercivity} at $t+\tau$ with
	$H=\qtwo(t+\tau)$ and
	$W=L(t+\tau,\cdot)+\mathbf R(t+\tau,\cdot)$, and observing that
	$\sigma(\Sn)=1$ implies
	$\|W\|_{L^1(\Sn)}\le S(t+\tau)+K_R$, we obtain
	\[
	\langle\qtwo(t+\tau),\Ftwo(t+\tau)\rangle
	\ge|\qtwo(t+\tau)|-2\kappa_n^{-1}
	\bigl(S(t+\tau)+K_R\bigr).
	\]
	Furthermore, \eqref{eq:q-Lipschitz} gives
	\begin{align*}
		|\qtwo(t+\tau)|
		&\ge|\qtwo(t)|-\frac{\kappa_n^{-1/2}}{n+2}D_n,
	\end{align*}
	while combining \eqref{eq:q-Lipschitz} with
	\eqref{eq:F2-uniform-bound}, we have
	\begin{align*}
		|\langle\qtwo(t)-\qtwo(t+\tau),\Ftwo(t+\tau)\rangle|
		&\le|\qtwo(t)-\qtwo(t+\tau)|\,|\Ftwo(t+\tau)|\\
		&\le\frac{\kappa_n^{-1}}{n+2}D_n.
	\end{align*}
	The inner product can be decomposed as
	\begin{align*}
		\langle\qtwo(t),\Ftwo(t+\tau)\rangle
		=\langle\qtwo(t+\tau),\Ftwo(t+\tau)\rangle +\langle\qtwo(t)-\qtwo(t+\tau),\Ftwo(t+\tau)\rangle.
	\end{align*}
	Combining the preceding inequalities, we obtain
	\begin{align}
		\langle\qtwo(t),\Ftwo(t+\tau)\rangle
		&\ge |\qtwo(t)|-\frac{\kappa_n^{-1/2}}{n+2}D_n
		-2\kappa_n^{-1}
		\bigl(e^{2D_n}\varepsilon_n|\qtwo(t)|+K_R\bigr)
		-\frac{\kappa_n^{-1}}{n+2}D_n\notag\\
		&\ge\frac34|\qtwo(t)|
		-\frac{D_n}{n+2}(\kappa_n^{-1/2}+\kappa_n^{-1})
		-2\kappa_n^{-1}K_R.\label{eq:positive-fixed-interval-pre}
	\end{align}
	The second inequality follows from \eqref{eq:epsilon-choice}, since
	\[
	2\kappa_n^{-1}e^{2D_n}\varepsilon_n|\qtwo(t)|
	\le\frac14|\qtwo(t)|.
	\]
	Once $D_n$ and $\varepsilon_n$ have been fixed, we choose $C_n^*$
	sufficiently large so that, whenever
	$|\qtwo(t)|\ge C_n^*(K_R+1)$, one has
	\[
	\frac{D_n}{n+2}
	\bigl(\kappa_n^{-1/2}+\kappa_n^{-1}\bigr)
	+2\kappa_n^{-1}K_R
	\le\frac14|\qtwo(t)|.
	\]
	Combining this bound with
	\eqref{eq:positive-fixed-interval-pre}, we obtain
	\begin{equation}\label{eq:positive-fixed-interval}
		\langle\qtwo(t),\Ftwo(t+\tau)\rangle
		\ge\frac12|\qtwo(t)|,
		\qquad 0\le\tau\le D_n.
	\end{equation}
	
	Rewriting \eqref{eq:qprime} with $s=t+\tau$ gives
	\[
	\qtwo'(t)=-\int_0^\infty e^{-\lambda\tau}\Ftwo(t+\tau)\,d\tau.
	\]
	Taking the inner product of this formula with $\qtwo(t)$, splitting the
	integral at $D_n$, and applying \eqref{eq:positive-fixed-interval} to the
	first part and \eqref{eq:F2-uniform-bound} to the second, we find that
	\begin{align*}
		\langle\qtwo(t),\qtwo'(t)\rangle
		&\le-\frac{|\qtwo(t)|}{2}
		\int_0^{D_n}e^{-\lambda\tau}\,d\tau
		+|\qtwo(t)|\kappa_n^{-1/2}
		\int_{D_n}^\infty e^{-\lambda\tau}\,d\tau\\
		&=-\frac{|\qtwo(t)|}{2\lambda}(1-e^{-\lambda D_n})
		+\frac{|\qtwo(t)|\kappa_n^{-1/2}}{\lambda}e^{-\lambda D_n}\\
		&\le-\frac{|\qtwo(t)|}{4\lambda}(1-e^{-\lambda D_n}),
	\end{align*}
	where the last inequality follows from \eqref{eq:D-choice}.  Since
	\eqref{eq:decrease-assumptions} ensures that
	$|\qtwo(t)|\ge(K_R+1)C_n^*>0$, the chain rule gives
	\[
	\frac d{dt}|\qtwo(t)|
	=\frac{\langle\qtwo(t),\qtwo'(t)\rangle}{|\qtwo(t)|}
	\le-\frac{1-e^{-\lambda D_n}}{4\lambda}.
	\]
	Thus \eqref{eq:descent} holds with
	\[
	c_{0,n}=\frac{1-e^{-\lambda D_n}}{4\lambda}>0.
	\]
\end{proof}

Lemmas~\ref{lem:remaining-change} and~\ref{lem:descent} yield the pointwise
Hessian bound at any center where the solution is nonzero.

\begin{proposition}[Uniform bound at a nonzero center]
	\label{prop:nonzero-center}
	Let $\bv$ be a weak solution of \eqref{eq:system} in $B_2\subset\R^n$,
	suppose
	$\|\bv\|_{L^\infty(B_2)}\le M$, and assume $\bv(0)\ne0$.  Then
	\begin{equation}\label{eq:nonzero-center-bound}
		|D^2\bv(0)|\le C_n(M+1).
	\end{equation}
\end{proposition}

\begin{proof}
	Apply Lemma~\ref{lem:decomposition}, and set
	\begin{equation*}
		K_R:=\max\left\{1,
		\sup_{t\ge1}\|\mathbf R(t,\cdot)\|_{L^\infty(\Sn)}\right\}.
	\end{equation*}
	Then $K_R\le C_n(M+1)$ by \eqref{eq:R-properties}, after increasing the
	constant $C_n$ if necessary, while the definition also ensures that
	$K_R\ge1$ and that \eqref{eq:R-uniform-bound} holds.  At $t=1$, the bounds in
	\eqref{eq:coefficient-bounds-basic}, together with the definition of $\bA$,
	give
	\begin{equation}\label{eq:q-start}
		|\qtwo(1)|\le C_n M.
	\end{equation}
	The assumption $\bv(0)\ne0$ yields $S(t)\to\infty$, whereas
	$\qtwo(t)\to\qtwo_\infty$.  Consequently, if
	\[
	S(1)\ge |\qtwo(1)|+K_R,
	\]
	Lemma~\ref{lem:remaining-change}, combined with \eqref{eq:q-start},
	immediately imply
	$|\qtwo_\infty|\le C_n(M+1)$.
	
	Otherwise, by continuity and the intermediate value theorem, there exists
	$t_0>1$ such that
	\begin{equation}\label{eq:S-equality}
		S(t_0)=|\qtwo(t_0)|+K_R.
	\end{equation}
	Indeed, in the case under consideration, the continuous function
	$S(t)-|\qtwo(t)|-K_R$ is negative at $t=1$ and tends to $+\infty$ as
	$S(t)\to\infty$ while $\qtwo(t)\to\qtwo_\infty$.  Let $\varepsilon_n$ and
	$$
	\mu_*:=(K_R+1)C_n^*
	$$
	be the quantities furnished by Lemma~\ref{lem:descent}, and choose $D_0>0$
	large enough that
	\begin{equation}\label{eq:D0-choice}
		e^{-D_0}\le\frac{\varepsilon_n}{8}.
	\end{equation}
	If $t_0\le1+D_0$, then \eqref{eq:q-Lipschitz} and
	\eqref{eq:q-start} yield $|\qtwo(t_0)|\le C_n(M+1)$.
	
	It therefore remains to consider $t_0>1+D_0$, for which we set
	$\tau_0=t_0-D_0$.  If
	\begin{equation}\label{eq:temporary-small-q}
		|\qtwo(t_0)|<4\left(K_R
		+\frac{\kappa_n^{-1/2}}{n+2}D_0+\mu_*\right),
	\end{equation}
	then the bounds $K_R,\mu_*\le C_n(M+1)$ show that
	$|\qtwo(t_0)|\le C_n(M+1)$.
	
	We may thus assume that the reverse of \eqref{eq:temporary-small-q} holds,
	and \eqref{eq:q-Lipschitz} then yields
	\begin{equation}\label{eq:q-backward-large}
		|\qtwo(\tau_0)|
		\ge|\qtwo(t_0)|-\frac{\kappa_n^{-1/2}}{n+2}D_0
		\ge\frac34|\qtwo(t_0)|\ge\mu_*.
	\end{equation}
	Combining \eqref{eq:S-growth}, \eqref{eq:S-equality},
	\eqref{eq:D0-choice}, and \eqref{eq:q-backward-large}, we also find that
	\begin{align}
		S(\tau_0)
		&\le e^{-D_0}S(t_0)
		=e^{-D_0}(|\qtwo(t_0)|+K_R)\notag\\
		&\le\frac{5\varepsilon_n}{32}|\qtwo(t_0)|
		\le\frac{5\varepsilon_n}{24}|\qtwo(\tau_0)|
		<\frac{\varepsilon_n}{4}|\qtwo(\tau_0)|,
		\label{eq:S-backward-small}
	\end{align}
	where the reverse of \eqref{eq:temporary-small-q} gives
	$K_R\le|\qtwo(t_0)|/4$.
	
	If $|\qtwo(\tau_0)|\le|\qtwo(1)|$, then
	\eqref{eq:q-Lipschitz} and \eqref{eq:q-start} give
	\[
	|\qtwo(t_0)|
	\le|\qtwo(\tau_0)|+\frac{\kappa_n^{-1/2}}{n+2}D_0
	\le|\qtwo(1)|+\frac{\kappa_n^{-1/2}}{n+2}D_0
	\le C_n(M+1),
	\]
	which is the required bound.
	
	To exclude the remaining possibility, suppose that
	$|\qtwo(\tau_0)|>|\qtwo(1)|$, in which case continuity of
	$s\mapsto|\qtwo(s)|$ ensures that it attains its maximum on $[1,\tau_0]$ at
	some $t_*$.  The strict inequality just assumed gives $t_*>1$, while
	the estimate \eqref{eq:q-backward-large}, the monotonicity of $S$, and
	\eqref{eq:S-backward-small} together yield
	\[
	|\qtwo(t_*)|\ge|\qtwo(\tau_0)|\ge\mu_*,
	\quad\text{and}\quad
	S(t_*)\le S(\tau_0)
	<\frac{\varepsilon_n}{4}|\qtwo(\tau_0)|
	\le\varepsilon_n|\qtwo(t_*)|.
	\]
	For every sufficiently small $h>0$, maximality implies
	\[
	\frac{|\qtwo(t_*)|-|\qtwo(t_*-h)|}{h}\ge0.
	\]
	As established after \eqref{eq:q-ODE}, $\qtwo\in C^{1,1}_{\mathrm{loc}}$,
	and since $|\qtwo(t_*)|>0$, the function $s\mapsto|\qtwo(s)|$ is therefore
	differentiable at $t_*$.  Letting $h\rightarrow0$ in the last inequality
	gives
	\[
	\left.\frac d{dt}|\qtwo(t)|\right|_{t=t_*}\ge0.
	\]
	Since Lemma~\ref{lem:descent} states that the same derivative is at most
	$-c_{0,n}<0$, the nonnegative derivative obtained above is impossible,
	which excludes the remaining possibility and yields
	$|\qtwo(\tau_0)|\le|\qtwo(1)|$.
	
	Therefore, in every case, we have
	\[
	|\qtwo(t_0)|
	\le C_n(M+1).
	\]
	Applying Lemma~\ref{lem:remaining-change} at the value $t_0$ determined by
	\eqref{eq:S-equality}, we conclude that
	\begin{equation*}
		|\qtwo_\infty|\le C_n(M+1).
	\end{equation*}
	Finally, \eqref{eq:q-Hessian-relation-2} establishes
	\eqref{eq:nonzero-center-bound}.
\end{proof}

\section{Proof of the optimal regularity estimate}\label{sec:final-proof}

Proposition~\ref{prop:nonzero-center} controls the Hessian on the open set
where the solution is nonzero.  The zero set is handled by the following
standard Sobolev result (see \cite[Sec.~4.2.2]{EG15}).

\begin{lemma}
	\label{lem:sobolev-level-set}
	Let $\Omega\subset\R^n$ be an open set, let $u\in W^{1,p}_{\mathrm{loc}}(\Omega)$ for some
	$p\ge1$, and let $c\in\R$.  Then
	\begin{equation}\label{eq:level-set-gradient-zero}
		Du=0\quad\text{almost everywhere on }\{u=c\}.
	\end{equation}
\end{lemma}

Applying the conclusion \eqref{eq:level-set-gradient-zero} in
Lemma~\ref{lem:sobolev-level-set} first to every component of the solution
and then to each first derivative shows that the Hessian vanishes almost
everywhere on the zero set.  On its complement, a rescaling of
Proposition~\ref{prop:nonzero-center} supplies the uniform estimate.

\begin{proof}[Proof of Theorem~\ref{thm:main}]
	Fix $x\in B_R(x_0)$ with $\bu(x)\ne0$ and define
	\begin{equation*}
		\bv(y)=\frac4{R^2}\bu\left(x+\frac R2y\right),
		\qquad y\in B_2.
	\end{equation*}
	If $|y|<2$, then
	\[
	\left|x+\frac R2y-x_0\right|
	\le|x-x_0|+\frac R2|y|<2R,
	\]
	so $x+(R/2)y\in B_{2R}(x_0)$.  The chain rule and the
	identity $\cN(\rho z)=\cN(z)$ for $\rho>0$ show that
	\[
	\Delta_y\bv(y)
	=\frac4{R^2}\left(\frac R2\right)^2
	\Delta_x\bu\left(x+\frac R2y\right)
	=\cN\!\left(\bu\left(x+\frac R2y\right)\right)
	=\cN(\bv(y)),
	\]
	so $\bv$ is a weak solution of the same equation, and
	\[
	\|\bv\|_{L^\infty(B_2)}
	\le\frac4{R^2}\|\bu\|_{L^\infty(B_{2R}(x_0))}.
	\]
	The second derivatives satisfy
	\[
	D_y^2\bv(0)
	=\frac4{R^2}\left(\frac R2\right)^2D_x^2\bu(x)
	=D_x^2\bu(x).
	\]
	Proposition~\ref{prop:nonzero-center} therefore yields the following estimate
	at every point of $B_R(x_0)\cap\{\bu\ne0\}$:
	\begin{equation}\label{eq:nonzero-uniform}
		|D^2\bu(x)|
		\le C_n\left(1+R^{-2}
		\|\bu\|_{L^\infty(B_{2R}(x_0))}\right).
	\end{equation}
	At every such point the Hessian in \eqref{eq:nonzero-uniform} is classical.
	Indeed, by \eqref{eq:finite-p} and the Sobolev embedding, $\bu$ has a
	continuous representative.  Since $\bu(x)\ne0$, there exists $\rho>0$ such
	that
	\[
	|\bu(y)|\ge\frac12|\bu(x)|>0
	\qquad\text{for every }y\in B_\rho(x).
	\]  On this
	neighborhood, $\cN$ is smooth near the range of $\bu$, and the interior
	Schauder estimate implies $\bu\in C^{2,\alpha}$ locally.
	
	It remains to identify the Hessian on the zero set.  Fix measurable
	representatives of the weak derivatives.  Apply
	Lemma~\ref{lem:sobolev-level-set} first to each component $u_\alpha$.
	Outside a null set,
	\[
	D_i u_\alpha=0\quad\text{on }\{\bu=0\},
	\qquad \alpha=1,\ldots,m,\quad i=1,\ldots,n,
	\]
	because $\{\bu=0\}\subset\{u_\alpha=0\}$.  By
	\eqref{eq:finite-p}, $D_i u_\alpha\in W^{1,p}_{\mathrm{loc}}$.  Applying the
	same lemma to every $D_i u_\alpha$ shows, outside another null set, that
	\[
	D_jD_i u_\alpha=0\quad\text{on }\{D_i u_\alpha=0\}.
	\]
	There are only finitely many indices because $n$ and $m$ are finite.  After
	combining their exceptional null sets, we obtain
	\begin{equation}\label{eq:Hessian-zero-set}
		D^2\bu=0\quad\text{almost everywhere on }\{\bu=0\}.
	\end{equation}
	Combining \eqref{eq:nonzero-uniform} and
	\eqref{eq:Hessian-zero-set} proves \eqref{eq:main-estimate}.
\end{proof}

\begin{remark}
	The estimate on the nonzero set follows from the matrix coefficient bound,
	while the zero set estimate follows from the Sobolev level set theorem.
	Therefore the proof requires neither a classification of free boundary
	points nor any regularity, density, or measure assumption on the free
	boundary.
\end{remark}

{\textbf{Acknowledgements}}\quad
DU is supported by National Nature Science Foundation of China under Grants 12125102, 12526202, 12671245 and 12631009. 
TANG is supported by National Nature Science Foundation of China under Grant 124B2012, Postdoctoral Fellowship Program and China Postdoctoral Science Foundation (BX20250058), and China Postdoctoral Science Foundation (2025M783083). WANG is supported by National Nature Science Foundation of China under Grant 12301258 and Fundamental Research Funds for the Central Universities 2682025CX060.

{\bf Data availability}\quad The authors confirm that this manuscript has no associated data.

{\bf Declarations}

{\bf Conflict of interest}\quad On behalf of all authors, the corresponding author states that there is no conflict of interest.


\begin{thebibliography}{ASUW15}
	
	\bibitem[ACF84]{ACF84}
	H.~W. Alt, L.~A. Caffarelli, and A.~Friedman,
	\emph{Variational problems with two phases and their free boundaries},
	Trans. Amer. Math. Soc. \textbf{282} (1984), 431--461.
	
	\bibitem[ALS13]{ALS13}
	J.~Andersson, E.~Lindgren, and H.~Shahgholian,
	\emph{Optimal regularity for the no-sign obstacle problem},
	Comm. Pure Appl. Math. \textbf{66} (2013), 245--262.
	
	\bibitem[ASUW15]{ASUW15}
	J.~Andersson, H.~Shahgholian, N.~N.~Uraltseva, and G.~S.~Weiss,
	\emph{Equilibrium points of a singular cooperative system with free boundary},
	Adv. Math. \textbf{280} (2015), 743--771.
	
	\bibitem[ASW10]{ASW10}
	J.~Andersson, H.~Shahgholian, and G.~S.~Weiss,
	\emph{Uniform regularity close to cross singularities in an unstable free
		boundary problem},
	Comm. Math. Phys. \textbf{296} (2010), 251--270.
	
	\bibitem[Caf77]{Caf77}
	L.~A.~Caffarelli,
	\emph{The regularity of free boundaries in higher dimensions},
	Acta Math. \textbf{139} (1977), 155--184.
	
	\bibitem[Caf80]{Caf80}
	L.~A.~Caffarelli,
	\emph{Compactness methods in free boundary problems},
	Comm. Partial Differential Equations \textbf{5} (1980), 427--448.
	
	\bibitem[Caf98]{Caf98}
	L.~A.~Caffarelli,
	\emph{The obstacle problem revisited},
	J. Fourier Anal. Appl. \textbf{4} (1998), 383--402.
	
	\bibitem[CLW26]{CLW26}
	S.~Chen, Y.~Li, and X.~Wang,
	\emph{Uniqueness of blow-ups for the superconductivity free boundary problem},
	arXiv:2604.23682, 2026.
	
	\bibitem[CSV18]{CSV18}
	M.~Colombo, L.~Spolaor, and B.~Velichkov,
	\emph{A logarithmic epiperimetric inequality for the obstacle problem},
	Geom. Funct. Anal. \textbf{28} (2018), 1029--1061.
	
	\bibitem[Dan20]{Dan20}
	D.~Danielli,
	\emph{An overview of the obstacle problem},
	Notices Amer. Math. Soc. \textbf{67} (2020), 1487--1497.
	
	\bibitem[DJS22]{DJS22}
	D.~De~Silva, S.~Jeon, and H.~Shahgholian,
	\emph{Almost minimizers for a singular system with free boundary},
	J. Differential Equations \textbf{336} (2022), 167--203.
	
	\bibitem[DJS23]{DJS23}
	D.~De~Silva, S.~Jeon, and H.~Shahgholian,
	\emph{Almost minimizers for a sublinear system with free boundary},
	Calc. Var. Partial Differential Equations \textbf{62} (2023),
	Paper No.~149, 43~pp.
	
	\bibitem[DJS26]{DJS26}
	D.~De~Silva, S.~Jeon, and H.~Shahgholian,
	\emph{The free boundary for a superlinear system},
	J. Funct. Anal. \textbf{290} (2026), 39~pp.
	
	
	\bibitem[EG15]{EG15}
	L.~C.~Evans and R.~F.~Gariepy,
	\emph{Measure theory and fine properties of functions},
	revised ed., Textbooks in Mathematics,
	CRC Press, Boca Raton, FL, 2015.
	
	\bibitem[FGKS24]{FGKS24}
	A.~Figalli, A.~Guerra, S.~Kim, and H.~Shahgholian,
	\emph{Constraint maps: singularities vs free boundaries},
	arXiv:2407.21128, 2024.
	
	\bibitem[FGKS25]{FGKS25}
	A.~Figalli, A.~Guerra, S.~Kim, and H.~Shahgholian,
	\emph{Constraint maps and free boundaries},
	Notices Amer. Math. Soc. \textbf{72} (2025), 494--503.
	
	\bibitem[FGKS26]{FGKS26}
	A.~Figalli, A.~Guerra, S.~Kim, and H.~Shahgholian,
	\emph{Constraint maps: insights and related themes},
	La Matematica \textbf{5} (2026), 27~pp.
	
	\bibitem[Fig18]{Fig18}
	 A. Figalli, \emph{Free boundary regularity in obstacle problems}, Journ\'{e}es \'{E}quations aux d\'{e}riv\'{e}es partielles, (2018), 24 pp.
	
	\bibitem[FKS24]{FKS24}
	A.~Figalli, S.~Kim, and H.~Shahgholian,
	\emph{Constraint maps with free boundaries: the obstacle case},
	Arch. Ration. Mech. Anal. \textbf{248} (2024), 36~pp.
	
	\bibitem[FS19]{FS19}
	A.~Figalli and J.~Serra,
	\emph{On the fine structure of the free boundary for the classical obstacle
		problem},
	Invent. Math. \textbf{215} (2019), 311--366.
	
	\bibitem[FK24]{FK24}
	M.~Fotouhi and H.~Koch,
	\emph{Higher regularity of the free boundary in a semilinear system},
	Math. Ann. \textbf{388} (2024), 3897--3939.
	
	\bibitem[FSW21]{FSW21}
	M.~Fotouhi, H.~Shahgholian, and G.~S.~Weiss,
	\emph{A free boundary problem for an elliptic system},
	J. Differential Equations \textbf{284} (2021), 126--155.
	
	\bibitem[Fre72]{Fre72}
	J.~Frehse,
	\emph{On the regularity of the solution of a second order variational
		inequality},
	Boll. Un. Mat. Ital. (4) \textbf{6} (1972), 312--315.
	
	\bibitem[IMN17]{IMN17}
	E.~Indrei, A.~Minne, and L.~Nurbekyan,
	\emph{Regularity of solutions in semilinear elliptic theory},
	Bull. Math. Sci. \textbf{7} (2017), 177--200.
	
	\bibitem[KN77]{KN77}
	D.~Kinderlehrer and L.~Nirenberg,
	\emph{Regularity in free boundary problems},
	Ann. Scuola Norm. Sup. Pisa Cl. Sci. (4) \textbf{4} (1977), 373--391.
	
	\bibitem[Koi21]{Koi21}
	S.~Koike,
	\emph{Regularity of solutions of obstacle problems---old \& new},
	in \emph{Nonlinear Partial Differential Equations for Future Applications},
	Springer Proc. Math. Stat., vol.~346,
	Springer, Singapore, 2021, 205--243.
	
	
	\bibitem[Mon03]{Mon03}
	R.~Monneau,
	\emph{On the number of singularities for the obstacle problem in two
		dimensions},
	J. Geom. Anal. \textbf{13} (2003), 359--389.
	
	\bibitem[PSU12]{PSU12}
	A.~Petrosyan, H.~Shahgholian, and N.~Uraltseva,
	\emph{Regularity of free boundaries in obstacle-type problems},
	Graduate Studies in Mathematics, vol.~136,
	American Mathematical Society, Providence, RI, 2012.
	
	\bibitem[RS19]{RS19}
	X.~Ros-Oton and J.~Serra,
	\emph{Understanding singularities in free boundary problems},
	Matematica, Cultura e Societ\`a \textbf{4} (2019), 107--118.
	
	\bibitem[Sha03]{Sha03}
	H.~Shahgholian,
	\emph{$C^{1,1}$ regularity in semilinear elliptic problems},
	Comm. Pure Appl. Math. \textbf{56} (2003), 278--281.
	
	\bibitem[SUW04]{SUW04}
	H.~Shahgholian, N.~N.~Uraltseva, and G.~S.~Weiss,
	\emph{Global solutions of an obstacle-problem-like equation with two phases},
	Monatsh. Math. \textbf{142} (2004), 27--34.
	
	\bibitem[SUW07]{SUW07}
	H.~Shahgholian, N.~N.~Uraltseva, and G.~S.~Weiss,
	\emph{The two-phase membrane problem---regularity of the free boundaries in
		higher dimensions},
	Int. Math. Res. Not. (2007), 16~pp.
	
	\bibitem[Ura01]{Ura01}
	N.~N.~Uraltseva,
	\emph{Two-phase obstacle problem},
	J. Math. Sci. (New York) \textbf{106} (2001), 3073--3077.
	
	\bibitem[Wei99]{Wei99}
	G.~S.~Weiss,
	\emph{A homogeneity improvement approach to the obstacle problem},
	Invent. Math. \textbf{138} (1999), 23--50.
	
\end{thebibliography}
\end{document}